\documentclass[reqno]{amsart}
\usepackage{amssymb}
\usepackage{amsfonts}
\usepackage{pdfsync}

\numberwithin{equation}{section}

\newtheorem{theorem}{Theorem}[section]
\newtheorem{proposition}[theorem]{Proposition}

\theoremstyle{definition}

\theoremstyle{definition} 
\newtheorem{remark}[theorem]{Remark}
\newtheorem{remarks}[theorem]{Remarks}

\newcommand{\bea}{\begin{eqnarray}}
\newcommand{\eea}{\end{eqnarray}}
\newcommand{\beas}{\begin{eqnarray*}}
\newcommand{\eeas}{\end{eqnarray*}}
\newcommand{\beq}{\begin{equation}}
\newcommand{\eeq}{\end{equation}}

\newcommand{\cA}{A}
\newcommand{\cB}{\mathcal B}

\newcommand{\cD}{\mathcal D}
\newcommand{\cE}{\mathcal E}

\newcommand{\cN}{\mathcal N}

\newcommand{\cMH}{\mathcal{MH}}
\newcommand{\cPM}{\mathcal{SM}}

\newcommand{\cSM}{\mathcal{SM}}

\newcommand{\R}{\mathbb R}

\newcommand{\LL}{\mathbb L}
\newcommand{\EE}{\mathbb E}
\newcommand{\FF}{\mathbb F}

\DeclareMathSymbol{\complement}{\mathord}{AMSa}{"7B}

\def\vv<#1>{\langle #1\rangle}
\def\Vv<#1>{\bigl\langle #1\bigr\rangle}

\begin{document}


\title[Incompressible Two-Phase Flows with Phase Transitions]
{On the Qualitative Behaviour of Incompressible Two-Phase Flows with Phase Transitions:\\
The Case of Equal Densities}

\author[J.~Pr\"uss]{Jan Pr\"uss}
\address{Institut f\"ur Mathematik \\
         Martin-Luther-Universit\"at Halle-Witten\-berg\\
         D-60120 Halle, Germany}
\email{jan.pruess@mathematik.uni-halle.de}

\author[G.~Simonett]{Gieri Simonett}
\address{Department of Mathematics\\
         Vanderbilt University \\
         Nashville, TN~37240, USA}
\email{simonett@math.vanderbilt.edu}

\author[R.~Zacher]{Rico Zacher}
\address{Institut f\"ur Mathematik \\
         Martin-Luther-Universit\"at Halle-Witten\-berg\\
         D-60120 Halle, Germany}
\email{rico.zacher@mathematik.uni-halle.de}


\thanks{This work was partially supported by a grant from the Simons Foundation (\#245959 to Gieri Simonett).}

\begin{abstract}
The study of the basic model for incompressible two-phase flows with phase transitions in the case of equal densities, initiated in the paper Pr\"uss,
Shibata, Shimizu, and Simonett \cite{PSSS11}, is continued here with a stability analysis of equilibria and results on asymptotic behaviour of global solutions. The results parallel those for the thermodynamically consistent Stefan problem with surface tension obtained in Pr\"uss, Simonett, and Zacher \cite{PSZ10}.
\end{abstract}

\maketitle

{\small\noindent
{\bf Mathematics Subject Classification (2010):}\\
Primary: 35R35, 35K55, 35B35, Secondary: 35Q30, 76D45, 80A22. \vspace{0.1in}\\
{\bf Key words:} Two-phase Navier-Stokes equations, surface tension, phase transitions, entropy, semiflow, stability, generalized principle of linearized stability, convergence to equilibria} \vspace{0.1in}\\



\section{Introduction}
In this paper we study a sharp interface model for two-phase flows with surface tension
undergoing phase transitions. The model is based on
conservation of mass, momentum and energy, and hence is physically exact. It further employs
the standard constitutive law of Newton for the stress tensor, Fourier's law for heat conduction,
and it is thermodynamically consistent.

Suppose that two fluids, fluid$_1$  and fluid$_2$, occupy the
regions $\Omega_1(t)$ and $\Omega_2(t)$, respectively,
with $\bar\Omega_1(t)\cup\bar\Omega_2(t)=\bar\Omega$.
Let $\Gamma(t)=\partial\Omega_1(t)$ be a sharp interface that separates the fluids.
Across the interface $\Gamma(t)$ certain physical parameters, such as the
density, viscosity, heat capacity and the heat conductivity, experience jumps. We assume that the interface is ideal in the sense that it is {\em immaterial,} which means that it has no capacity for mass or energy except surface tension.

In more detail, let $\Omega\subset \R^{n}$ be a bounded domain of class $C^{3-}$ with $n\geq2$.
We further assume that $\Gamma(t)\cap\partial \Omega=\emptyset$, which implies that
no boundary contact can occur.
In the following we let
\begin{itemize}
\item $u_i$ denote the velocity field in $\Omega_i(t)$,
\item $\pi_i$ the pressure field in $\Omega_i(t)$,
\item $T_i$ the stress tensor in $\Omega_i(t)$,
\item  $D_i=(\nabla u_i +[\nabla u_i]^{\sf T})/2$ the rate of strain tensor in $\Omega_i(t)$,
\item $\theta_i$  the (absolute) temperature field in $\Omega_i(t)$,
\item $\nu_\Gamma$ the outer normal  of $\Omega_1(t)$,
\item  $u_\Gamma$ the velocity field of $\Gamma(t)$,
\item  $V_\Gamma=u_\Gamma\cdot\nu_\Gamma$ the normal velocity of $\Gamma(t)$,
\item  $H_\Gamma=H(\Gamma(t))=-{\rm div}_\Gamma \nu_\Gamma$ the sum of the principal curvatures of
$\Gamma(t)$, and
\item  $[\![v]\!]=v_2-v_1$ the jump of a quantity $v$ across $\Gamma(t)$.
\end{itemize}
Here  the sign of the curvature $H_\Gamma$ is negative at a point $x\in\Gamma$ if
$\Omega_1\cap B_r(x)$ is convex, for some sufficiently small $r>0$. Thus if $\Omega_1$ is a ball,
i.e.\ $\Gamma=S_R(x_0)$, then $H_\Gamma=-(n-1)/R$.

Several quantities are derived from the  {\em specific free energies} $\psi_i(\theta)$ as follows:
\begin{itemize}
\item $\epsilon_i(\theta)= \psi_i(\theta)+\theta\eta_i(\theta)$ is the
internal energy in phase $i$.
\item $\eta_i(\theta) =-\psi_i^\prime(\theta)$ is the entropy,
\item $\kappa_i(\theta)= e^\prime_i(\theta)=-\theta\psi_i^{\prime\prime}(\theta)>0$ is the  heat capacity,
\item $l(\theta)=\theta[\![\psi^\prime(\theta)]\!]=-\theta[\![\eta(\theta)]\!]$ is the latent heat.
\end{itemize}
Further $d_i(\theta)>0$ denotes the coefficient of heat conduction in Fourier's law, $\mu_i(\theta)>0$ the viscosity in Newton's law, $\rho:=\rho_1=\rho_2=1$ the constant density,
and $\sigma>0$ the constant coefficient of surface tension.

In the sequel we drop the index $i$, as there is no danger of confusion; we just keep in mind that the
physical quantities depend on the phases.

By  the {\em Incompressible two-phase flow with phase transition} we mean the following
free boundary problem:
find a family of closed compact hypersurfaces $\{\Gamma(t)\}_{t\geq0}$ contained in $\Omega$
and appropriately smooth functions $u:\R_+\times \bar{\Omega} \to \R^n$, and $\pi,\theta:\R_+\times\bar{\Omega}\rightarrow\R$ such that
\begin{equation}
\label{i2pp}
\left\{
\begin{aligned}
\partial_t u +u\cdot\nabla u -{\rm div}\, T &=0 && \text{in}&&\Omega\setminus\Gamma(t)\\
T=\mu(\theta)(\nabla u + [\nabla u]^{\sf T}) -\pi I,\quad
{\rm div }\, u &=0 &&\text{in}&&\Omega\setminus\Gamma(t)\\
\kappa (\theta)(\partial_t \theta+u\cdot\nabla\theta)-{\rm div}(d(\theta)\nabla \theta)
-T:\nabla u &=0 &&\text{in}&&\Omega\setminus\Gamma(t)\\
u=\partial_\nu \theta &=0 && \text{on}&&\partial \Omega\\
[\![u]\!]=[\![\theta]\!]&=0 && \text{on}&& \Gamma(t)\\
[\![T\nu_\Gamma]\!]+\sigma H_\Gamma \nu_\Gamma&=0 &&\text{on}&&\Gamma(t)\\
[\![\psi(\theta)]\!]+\sigma H_\Gamma &=0 &&\text{on}&&\Gamma(t)\\
-[\![d(\theta)\partial_\nu \theta]\!]+ l(\theta) (V_\Gamma- u\cdot \nu_\Gamma)&=0 &&\text{on}&&\Gamma(t)\\
 \Gamma(0)=\Gamma_0,\quad u(0,x)=u_0(x), \quad \theta(0,x)&=\theta_0(x) &&\text{in}&&\Omega.
\end{aligned}\right.
\end{equation}
This model  has been recently proposed by Anderson et al.\ \cite{Gur07},
see also the monographs by Ishii \cite{Ish75} and Ishii and Takashi~\cite{IsTa06},
and the derivation in Section 2 of the recent paper \cite{PSSS11}.
It has been shown in \cite{PSSS11} that the model is thermodynamically consistent in the sense that in
the absence of exterior forces and external heat sources, the total energy is preserved and the total entropy is nondecreasing. It is in some sense the simplest sharp interface model for incompressible Newtonian two-phase flows taking into account phase transitions driven by temperature.

There is a large literature on isothermal incompressible Newtonian two-phase flows without phase transitions, and also on the two-phase Stefan problem with surface tension modeling temperature driven phase transitions. On the other hand, mathematical work on two-phase flow problems including phase transitions are rare. In this direction, we only know the papers by Hoffmann and Starovoitov \cite{HoSt98a,HoSt98b} dealing with a simplified two-phase flow model, and Kusaka and Tani \cite{KuTa99,KuTa02} which is two-phase for temperature but only one phase is moving. The papers of DiBenedetto and Friedman \cite{DBFr86} and DiBenedetto and O'Leary\cite{DBOL93} deal with weak solutions of conduction-convection problems with phase change. However, none of these papers deals with models which are consistent with thermodynamics.

It is the purpose of this paper to present a qualitative analysis of problem \eqref{i2pp} in the framework of $L_p$-theory. We  discuss the induced local semiflow and study the stability properties of the equilibria. These are the same as those for the thermodynamically consistent two-phase Stefan problem with surface tension, and even more, also their stability properties turn out to be the same. This heavily depends on the fact that the densities of the two phases are assumed to be equal; in this case the problem is {\em temperature dominated}.

In a forthcoming paper we will consider the case where the densities are not equal; then the
solution behavior is different, as the interfacial mass flux has a direct impact on the velocity field of the fluid, inducing so-called {\em Stefan currents}. The velocity field is no longer continuous across the interface which leads to  different analytic properties of the model. We call this case {\em velocity dominated}.

It has been shown in \cite{PSSS11} that the total energy
\begin{equation}
\label{energy}
{\sf E}:={\sf E}(u,\theta,\Gamma):=\frac{1}{2}\int_{\Omega\setminus\Gamma} |u|_2^2\,dx
+ \int_{\Omega\setminus\Gamma} \epsilon(\theta)\,dx +\sigma |\Gamma|,
\end{equation}
is preserved along smooth solutions, while the total entropy
\begin{equation}
\label{entropy}
\Phi(\theta,\Gamma)=\int_{\Omega\setminus\Gamma} \eta(\theta)\,dx
\end{equation}
is strictly increasing along nonconstant smooth solutions. By similar arguments as in \cite{PSZ10}, it can further be shown that the equilibria
of (\ref{i2pp}) are precisely the critical points of the entropy functional with
prescribed energy, and that a necessary condition for such a point $e_*=(u_*,\theta_*,\Gamma_*)$ to be a local maximum of the entropy functional
with prescribed energy is that $\Gamma_*$ is connected and that the stability
condition {\bf (S)}, see Theorem \ref{linstabtheorem} below, is satisfied.

The plan for this paper - which builds on \cite{PSSS11} and \cite{PSZ10} - is as follows.
Our approach is based on the so-called {\em direct mapping method} where the problem with moving interface is transformed to a problem with fixed
domain, resulting in a quasilinear parabolic evolution problem with a dynamic boundary condition on a domain with fixed interface. The main result on well-posedness of the transformed problem is taken from  \cite{PSSS11}
and is stated in Section 2.  The linear stability properties of the equilibria are derived in Section 3. It turns out that generically the equilibria are normally hyperbolic. They are always unstable if the disperse phase $\Omega_1$ is not connected. If both phases are connected we find the same stability condition,
 condition {\bf (S)} in Theorem \ref{linstabtheorem} below, as in Pr\"uss, Simonett and Zacher \cite{PSZ10}, see also Pr\"uss and Simonett \cite{PrSi08}.  As the equilibria are normally hyperbolic we may use a variant of the generalized principle of linearized stability, see Pr\"uss, Simonett and Zacher \cite{PSZ09}, to prove nonlinear stability or instability. Combining this method with the Lyapunov functional we are able to show that a solution which does not develop singularities exists globally and its orbit is relatively compact in the state manifold. If such a solution contains a stable equilibrium in its limit set, then it is shown that it converges to this equilibrium.
\bigskip\\

\section{The Local Semiflow}
\noindent
{\bf (i) Local Existence}\\
The basic result for local well-posedness of problem \eqref{i2pp}
in an $L_p$-setting, stated in \cite[Theorem 5.1]{PSSS11}, is the following. Here $P_\Gamma=I-\nu_\Gamma\otimes \nu_\Gamma$ denotes the orthogonal projection onto the tangent space of
$\Gamma$.
\begin{theorem}
\label{wellposed}
Let $p>n+2$,  $\sigma>0$.
Suppose $\psi_i\in C^3(0,\infty)$, $\mu_i,d_i\in C^2(0,\infty)$ such that
$$\kappa_i(s)=-s\psi_i^{\prime\prime}(s)>0,\quad \mu_i(s)>0,\quad  d_i(s)>0,\quad s\in(0,\infty),\;
i=1,2.$$
Assume the {\bf regularity conditions}
$$ (u_0,\theta_0)\in [W^{2-2/p}_p(\Omega\setminus\Gamma_0)\cap C(\bar{\Omega})]^{n+1},\quad \Gamma_0\in W^{4-3/p}_p,$$
where $\Omega\subset\R^n$ is a bounded domain with boundary $\partial\Omega\in C^{3-}$,\\
 the {\bf  compatibility conditions}
\begin{align*}
&{\rm div}\, u_0=0 \; \mbox{ in } \Omega\setminus\Gamma_0,\quad u_0=\partial_\nu \theta_0 = 0 \mbox{ on } \partial\Omega,\\
&P_{\Gamma_0}[\![\mu(\theta_0)(\nabla u_0+[\nabla u_0]^{\sf T})]\!] =0 \; \mbox{ on } \Gamma_0,\\
&[\![\psi(u_0)]\!]+\sigma H_{\Gamma_0}=0\;\mbox{ on } \Gamma_0,
\quad [\![d(\theta_0)\partial_{\nu_{\Gamma_0}} u_0]\!]\in W^{2-6/p}_p(\Gamma_0),
\end{align*}
and the {\bf well-posedness condition}
$$\theta_0>0 \mbox{ on }  \bar\Omega, \quad  l(u_0)\neq0 \mbox{ on } \Gamma_0.$$

Then there exists a unique $L_p$-solution of problem \eqref{i2pp}  on
some possibly small but nontrivial time interval $J=[0,\tau]$.
\end{theorem}
\noindent{\bf (ii) The Local Semiflow}\\
We follow here the approach introduced in K\"ohne, Pr\"uss and Wilke \cite{KPW10} for the isothermal incompressible two-phase Navier-Stokes problem without phase transitions and in Pr\"uss, Simonett and  Zacher \cite{PSZ10} for the Stefan problem with surface tension.

Recall that the closed $C^2$-hypersurfaces contained in $\Omega$ form a $C^2$-manifold,
which we denote by $\cMH^2(\Omega)$.
The charts are the parameterizations over a given hypersurface $\Sigma$, and the tangent
space consists of the normal vector fields on $\Sigma$.
We define a metric on $\cMH^2(\Omega)$ by means of
$$d_{\cMH^2}(\Sigma_1,\Sigma_2):= d_H(\cN^2\Sigma_1,\cN^2\Sigma_2),$$
where $d_H$ denotes the Hausdorff metric on the compact subsets of $\R^n$ and
$\cN^2\Sigma=\{p,\nu_\Sigma(p),\nabla_\Sigma \nu_\Sigma(p)):\,p\in \Sigma\}$ the second order bundle
of $\Sigma\in \cMH^2(\Omega)$. This way $\cMH^2(\Omega)$ becomes a Banach manifold of class $C^2$.
\medskip

As an ambient space for the
state manifold $\cSM$ of problem  \eqref{i2pp}  we consider
the product space $C(\bar{\Omega})^{n+1}\times \cMH^2(\Omega)$,
due to continuity of velocity, temperature and curvature.

We  then define the state manifold $\cSM$ as follows.
\begin{eqnarray}
\label{phasemanif0}
\cSM:=&&\hspace{-0.5cm}\Big\{(u,\theta,\Gamma)\in C(\bar{\Omega})^{n+1}\times \cMH^2:
 (u,\theta)\in W^{2-2/p}_p(\Omega\setminus\Gamma)^{n+1},\, \Gamma\in W^{4-3/p}_p,\nonumber\\
 && {\rm div}\, u=0\; \mbox{ in }\; \Omega,\quad \theta>0 \mbox{ in } \bar{\Omega},\;\quad u=\partial_\nu\theta =0\; \mbox{ on } \partial\Omega,\nonumber\\
&& P_\Gamma[\![\mu(\theta) D\nu_\Gamma]\!] =0,\quad [\![\psi(\theta)]\!]+\sigma H_\Gamma=0
\mbox{ on } \Gamma,\\
&&l(\theta)\neq0\mbox{ on } \Gamma,
\quad [\![d\partial_\nu \theta]\!] \in W^{2-6/p}_p(\Gamma)\Big\},\nonumber
\end{eqnarray}
Charts for these manifolds are obtained by the charts induced by $\cMH^2(\Omega)$,
followed by a Hanzawa transformation.

Applying Theorem \ref{wellposed} and re-parameterizing the interface repeatedly,
we see that (\ref{i2pp}) yields a local semiflow on $\cSM$.

\begin{theorem}
\label{semiflow}
Let $p>n+2$, $\sigma>0$, and suppose $\psi_i\in C^3(0,\infty)$, $\mu_i,d_i\in C^2(0,\infty)$ such that
$$\kappa_i(s)=-s\psi_i^{\prime\prime}(s)>0,\quad \mu_i(s)>0,\quad  d_i(s)>0,\quad s\in(0,\infty),
\; i=1,2.$$
Then problem (\ref{i2pp}) generates a local semiflow
on the state manifold $\cSM$. Each solution $(u,\theta,\Gamma)$ exists on a maximal time
interval $[0,t_*)$, where $t_*=t_*(u_0,\theta_0,\Gamma_0)$.
\end{theorem}
Note that the pressure does not occur explicitly as a variable in the
local semiflow, as the latter is only formulated in terms of the temperature $\theta$,
the velocity field $u$, and the free boundary $\Gamma$.
The pressure $\pi$ is determined for each time $t$ from
$(u,\theta,\Gamma)$ by means of the weak transmission problem
\begin{align*}
\big(\nabla\pi|\nabla\phi\big)_{L_2(\Omega)} &= \big(2{\rm div}(\mu(\theta)D)-u\cdot\nabla u|\nabla\phi\big)_{L_2(\Omega)},\quad \phi\in H^1_{p^\prime}(\Omega),\\
[\![\pi]\!] &= \sigma H_\Gamma + 2[\![\mu(\theta)D\nu_\Gamma\cdot \nu_\Gamma]\!]\quad \mbox{ on } \Gamma.
\end{align*}
Concerning such transmission problems we refer to \cite[Scetion 8]{KPW10}.

\section{Linear Stability of Equilibria}
\noindent
{\bf 1.}\,
As shown in \cite[Section 3]{PSSS11}, the equilibria $(u_*,\pi_*,\theta_*,\Gamma_*)$ of \eqref{i2pp} consist of zero velocities $u_*$, constant pressures $\pi_*$ in the phases,
constant temperatures $\theta_*$, and $\Omega_1$ is a ball $\Omega_1=B_{R_*}(x_*)\subset\Omega$ in case $\Omega_1$ is connected, and  a union of nonintersecting balls of equal radii otherwise.
We assume here that the balls do not touch the outer boundary $\partial\Omega$, to avoid the contact angle problem, and we also assume that the balls do not touch each other. We are not able to handle the latter case as the interface $\Gamma_*=\partial\Omega_1$ will then not be a $C^2$-manifold.
We call such equilibria {\em non-degenerate}.
The temperature $\theta_*$ and the pressure jump $[\![\pi_*]\!]$ are related to $R_*$
via the curvature $H_{\Gamma_*}$ through the relation
\begin{equation}
\label{pi-H}
[\![\psi(\theta_*)]\!]=-\sigma H_{\Gamma_*}= \frac{(n-1)\sigma}{R_*},
\quad [\![\pi_*]\!]=-[\![\psi(\theta_*)]\!].
\end{equation}
In the sequel we only consider non-degenerate equilibria and denote the set of such equilibria
by $\cE$, i.e.,
\begin{equation*}
\cE=\left\{(0,\theta_*,\Gamma_*):\Gamma_* =\bigcup_{k=1}^m \Gamma^k_*,
\quad \Gamma^k_*=S_{R_\ast}(x^k_*)\right\},
\end{equation*}
with $[\![\pi_*]\!],$ $\theta_*$ and $R_*$ determined by \eqref{pi-H}.
According to \eqref{energy} the
total energy at an equilibrium $(0,\theta_*,\Gamma_*)$ is then given by
\begin{equation}
\label{energy-equilibrium}
\varphi(\theta_*):=E(0,\theta_*,\Gamma_*)=
\int_{\Omega\setminus\Gamma_*}\epsilon(\theta_*)\,dx +\sigma|\Gamma_*|.
\end{equation}

By employing the Hanzawa transformation, see \cite[Section 2]{PSSS11}, one shows that
the fully linearized problem at an equilibrium is given by
\begin{equation}
\label{elin-u-theta}
\left\{
\begin{aligned}
\partial_t u-\mu_* \Delta u +\nabla\pi & =f_u &&\text{in} && \Omega\setminus{\Gamma_*}\\
{\rm div}\, u &=f_d &&\text{in} &&\Omega\setminus{\Gamma_*}\\
u &=0 && \text{on} &&\partial\Omega\\
[\![u]\!] &= 0 &&\text{on} && {\Gamma_*}\\
 -[\![T\nu]\!] +\sigma (\cA_* h)\nu &=g_u &&\text{on} &&{\Gamma_*}\\
\kappa_*\partial_t\vartheta -d_*\Delta \vartheta &=f_\theta &&\text{in} && \Omega\setminus{\Gamma_*}\\
\partial_\nu\vartheta &=0 &&\text{on} && \partial\Omega\\
[\![\vartheta]\!]& =0 &&\text{on} &&{\Gamma_*} \\
 l_*\vartheta -\sigma\cA_*h &= g_\theta &&\text{on} && {\Gamma_*}\\
(l_*/\theta_*)(\partial_t h-u\cdot\nu) -[\![d_*\partial_\nu\vartheta]\!] &= g_h
&&\text{on} && {\Gamma_*}\\
u(0)=u_0,\;\vartheta(0)=\vartheta_0,\: h(0)&=h_0, && &&
\end{aligned}
\right.
\end{equation}
with
\begin{equation}
\label{H-prime}
\mu_*=\mu(\theta_*),\;\; \kappa_*=\kappa(\theta_*), \;\; l_*=l(\theta_*),
\;\;\cA_* =-H^\prime(0)=-{(n-1)}/{R_*^2} -\Delta_{\Gamma_*},
\end{equation}
 where $\Delta_{\Gamma_*}$ denotes the Laplace-Beltrami operator of $\Gamma_*$,
and  $\vartheta=(\theta-\theta_*)/\theta_*$ is the relative temperature.
\medskip\\
\noindent
It follows from the maximal regularity results in \cite{PSSS11} that the operator $\LL$ defined by the left hand side of \eqref{elin-u-theta}
is an isomorphism from $\EE(J)$ into $\R(J)\subset\FF(J)\times \gamma\EE$,
where $\R(J)$ is determined by the natural compatibility conditions.
Here the function spaces $\EE(J)$, $\gamma\EE(J)$ and $\FF(J)$,
with $J=[0,a]$ an interval, are defined as follows:
$$
\EE(J):= \big\{(u,\pi,q,\theta,h)\in
\EE_1(J) \times \EE_2(J) \times \EE_3(J) \times \EE_4(J) \times
\EE_5(J): q=[\![\pi]\!]\big\},
$$
where
{\allowdisplaybreaks
\begin{align*}
&\EE_1(J):= \{u\in H_p^1(J; L_p(\Omega))^n\cap
L_p(J; H_p^2(\Omega\setminus \Gamma_*))^n :\, u=0\ {\rm on}\ \partial\Omega,
\ \ [\![u]\!]=0\},\\
&\EE_2(J):= L_p(J; \dot H^1_p(\Omega\setminus \Gamma_*)),\\
&\EE_3(J):= W_p^{1/2-1/2p}(J; L_p(\Gamma_*))
\cap L_p(J; W^{1-1/p}_p(\Gamma_*)),\\
&\EE_4(J):= \{\theta\in H_p^1(J; L_p(\Omega))\cap
L_p(J; H^2_p(\Omega\setminus \Gamma_*)):\, \partial_\nu\theta=0\
{\rm on}\ \partial\Omega,\ \ [\![\theta]\!]=0\},\\
&\EE_5(J):= W_p^{3/2-1/2p}(J; L_p(\Gamma_*))
\cap W_p^{1-1/2p}(J; H^2_p(\Gamma_*))
\cap L_p(J; W^{4-1/p}_p(\Gamma_*)).
\end{align*}}
\noindent
The time-trace space $\gamma\EE(J)$  is given by
\begin{equation*}
\gamma\EE(J)=\big\{(u_0,\vartheta_0,h_0)\in
\big(W^{2-2/p}_p(\Omega\setminus\Gamma_*)\cap C(\bar{\Omega})\big)^{n+1}
\times W^{4-3/p}_p(\Gamma_*)\big\},
\end{equation*}
while the space of data is
$
\FF(J):= \big\{(f_u,f_d,g_u,f_\theta, g_\theta,g_h)\in \prod_{j=1}^6 \FF_j(J)\big\},
$
where
{\allowdisplaybreaks
\begin{align*}
&\FF_1(J):= L_p(J; L_p(\Omega))^n,\\
&\FF_2(J):= H_p^{1}(J; \dot H_p^{-1}(\Omega)) \cap L_p(J; H_p^1(\Omega\setminus\Gamma_*)),\\
&\FF_3(J):= W_p^{1/2-1/2p}(J; L_p(\Gamma_*))^n \cap L_p(J; W^{1-1/p}_p(\Gamma_*))^n,\\
&\FF_4(J):= L_p(J; L_p(\Omega)),\\
&\FF_5(J):= W_p^{1-1/2p}(J; L_p(\Gamma_*))\cap L_p(J; W^{2-1/p}_p(\Gamma_*)),\\
&\FF_6(J):= W_p^{1/2-1/2p}(J; L_p(\Gamma_*))\cap L_p(J; W^{1-1/p}_p(\Gamma_*)).
\end{align*}}
If the time derivatives $\partial_t$ are replaced by $\partial_t +\omega$, with $\omega>0$ sufficiently large, then this result is also true for $J=\R_+$.
\medskip\\
\noindent
{\bf 2.}\,
 We introduce a functional analytic setting as follows.
Set $$X_0=L_{p,\sigma}(\Omega)^n\times L_p(\Omega)\times W^{2-2/p}_p(\Gamma_*),$$
where the subscript $\sigma$ means solenoidal, and define the operator $L$ by
\begin{equation*}
 L(u,\theta,h)=
\big(-\mu_*\Delta u +\nabla\pi, -(d_*/\kappa_*)\Delta \vartheta,
-u\cdot\nu-(\theta_*/l_*)[\![d_*\partial_{\nu}\vartheta]\!]\big).
\end{equation*}
To define the domain $D(L)$ of $L$, we set
\begin{eqnarray*}
X_1= \big\{(u,\vartheta,h)\in \big(H^2_p(\Omega\setminus\Gamma_*)\cap C(\bar\Omega)\big)^{n+1}\times W^{4-1/p}_p(\Gamma_*): \\
\quad {\rm div}\, u=0\; \mbox{ in }\; \Omega\setminus\Gamma_*,\;\:
u=\partial_\nu\vartheta=0\;\mbox{ on }\;\partial\Omega\big\},
\end{eqnarray*}
and
\begin{eqnarray*}
D(L)= \big\{(u,\vartheta,h)\in X_1:  [\![P_{*}\mu_* D\nu]\!]=0,\;\;
l_*\vartheta -\sigma\cA_* h=0\mbox{ on } \; \Gamma_*, \\
\,[\![d_*\partial_{\nu}\vartheta]\!] \in W^{2-2/p}_p(\Gamma_*) \big\},
\end{eqnarray*}
where $P_*=P_{\Gamma_*}$ denotes the orthogonal projection onto the tangent space of $\Gamma_*$.
Here $\pi$ is determined as the solution of the weak transmission problem
\begin{equation*}
\begin{split}
(\nabla\pi|\nabla\phi)_2 &=(\mu_*\Delta u|\nabla \phi)_2,\quad \phi\in \dot{H}^1_{p^\prime}(\Omega),\\
\qquad [\![\pi]\!] &=-\sigma \cA_*h+2[\![\mu_* D\nu\cdot\nu]\!].
\end{split}
\end{equation*}
We refer to \cite[Section 8]{KPW10} for a detailed analysis of such transmission problems.
\\
The linearized problem can be rewritten as an abstract evolution problem in $X_0$,
\begin{equation}\label{alp} \dot{z} + Lz =f,\quad t>0,\quad z(0)=z_0,\end{equation}
where $z=(u,\vartheta,h)$, $f=(f_u,f_\theta, g_h)$, $z_0=(u_0,\vartheta_0,h_0)$,
provided $f_d=g_u=g_\theta=0$.
As the terms $u\cdot \nu$ and $\sigma A_*h$ are of lower order
we may deduce maximal $L_p$-regularity of (\ref{alp}) from that of the decoupled system (cf.\ \cite[Section 6]{KPW10} and \cite[Theorem 4.3]{PSZ10}) by a perturbation argument.
We can then conclude that $-L$ generates an analytic $C_0$-semigroup in $X_0$;
cf.\ Pr\"uss \cite[Proposition 1.1]{Pru03}.
\medskip\\
\noindent
{\bf 3.\,The eigenvalue problem.}
Since the embedding $X_1\hookrightarrow X_0$ is compact, the semigroup $e^{-Lt}$ as well as the resolvent $(\lambda+L)^{-1}$ of $-L$ are compact. Therefore, the spectrum $\sigma(L)$ of $L$ consists only of countably many eigenvalues of finite algebraic multiplicity and is independent of $p\in (1,\infty)$.
Therefore it is enough to consider the case $p=2$.
In the following, we will use the notation
\begin{equation*}
\begin{split}
&(u|v)_\Omega:=(u|v)_{L_2(\Omega)}:=\int_\Omega u\bar v\,dx,
\quad u,v\in L_2(\Omega),\\
&(g|h)_{\Gamma_*}\!:=(g|h)_{L_2(\Gamma_*)}\!:=\int_{\Gamma_*} g\bar h\,ds,
\quad g,h\in L_2(\Gamma_*),
\end{split}
\end{equation*}
for the $L_2$ inner product in $\Omega$ and $\Gamma_*$, respectively.
Moreover, we set $|v|_\Omega=(v|v)^{1/2}_{\Omega}$ and
$|g|_{\Gamma_*}=(g|g)^{1/2}_{\Gamma_*}$.
The eigenvalue problem for $-L$ reads as follows:
\begin{equation}
\label{ev-u}
\left\{
\begin{aligned}
\lambda u-\mu_* \Delta u +\nabla\pi &=0 &&\text{in} && \Omega\setminus{\Gamma_*}\\
{\rm div}\, u &=0 &&\text{in} &&\Omega\setminus{\Gamma_*}\\
u &=0 && \text{on} &&\partial\Omega\\
[\![u]\!] &= 0 &&\text{on} && {\Gamma_*}\\
 -[\![T\nu]\!] +\sigma (\cA_* h)\nu &= 0 &&\text{on} &&{\Gamma_*},\\
\end{aligned}
\right.
\end{equation}
\smallskip
\begin{equation}
\hspace{-1.7cm}
\label{ev-theta}
\left\{
\begin{aligned}
\kappa_*\lambda\vartheta -d_*\Delta \vartheta &= 0 &&\text{in} && \Omega\setminus{\Gamma_*}\\
\partial_\nu\vartheta &=0 &&\text{on} && \partial\Omega\\
[\![\vartheta]\!]& =0 &&\text{on} &&{\Gamma_*} \\
 l_*\vartheta -\sigma\cA_*h &= 0 &&\text{on} && {\Gamma_*}\\
(l_*/\theta_*)(\lambda h-u\cdot\nu) -[\![d_*\partial_\nu\vartheta]\!] &=0 &&\text{on} && {\Gamma_*}.
\end{aligned}
\right.
\end{equation}
\goodbreak
\noindent
We are now ready to formulate the main result of this section.
\begin{theorem} \label{linstabtheorem}
Let $L$ denote the linearization at  $(0,\theta_*,\Gamma_*)\in\cE$ as defined above. Suppose $l_*\neq0$.
Then $-L$ generates a compact analytic $C_0$-semigroup in $X_0$ which has maximal $L_p$-regularity. The spectrum of $L$ consists only of eigenvalues of finite algebraic multiplicity. Moreover, the following assertions are valid.
\vspace{2mm}
\begin{itemize}
\item[{\bf (i)}] If $\Gamma_*$ is connected and the stability condition
\begin{equation*}
{\bf (S)}\qquad s:=s(e_*):=\frac{\sigma(n-1)}{R_*^2}-\frac{l_*^2|\Gamma_*|}{\theta_*(\kappa_*|1)_\Omega}
 \le 0
\end{equation*}
holds, then all eigenvalues $\lambda\neq 0$ of $-L$ have negative real part.
\vspace{2mm}
\item[{\bf (ii)}] The stability condition (S) is equivalent to $\varphi^\prime(\theta_*)\leq0$, where the function $\varphi$ is defined in \eqref{energy-equilibrium}.
\vspace{2mm}
\item[{\bf (iii)}] If $\Gamma_*=\cup_{k=1}^m\Gamma^k_*$ and $s>0$, then $-L$ has precisely $m$ positive eigenvalues, and  precisely $(m-1)$ positive eigenvalues if $s\le 0$.
\vspace{2mm}
\item[{\bf (iv)}] $\lambda=0$ is an eigenvalue of $L$ with geometric multiplicity $(mn+1)$.
It is semi-simple if $s\neq 0$.
\vspace{2mm}
\item[{\bf (v)}] Let $e_*=(0,\theta_*,\Gamma_*)\in \cE$ be an equilibrium. Then
in a neighborhood of $e_*$ the set of equilibria $\cE$ forms a $(mn+1)$-dimensional $C^1$-manifold.
Moreover, the kernel $N(L)$ of $L$ is isomorphic to the tangent space $T_{e_*}\cE$ of $\cE$ at $e_*$.
\end{itemize}
\vspace{2mm}
Consequently, $(0,\theta_*,\Gamma_*)\in\cE$ is normally stable if and only if $s<0$ and $\Gamma_*$ is connected, and normally hyperbolic if and only if $s>0$, or $\Gamma_*$ is disconnected and $s\neq 0$.
\end{theorem}
\smallskip
\begin{proof}
\noindent
{\bf (i)} Suppose  $\lambda$ with ${\rm Re}\; \lambda\geq0$ is an eigenvalue of $-L$
with eigenfunction $(u,\vartheta, h)$. Taking the inner product of the eigenvalue problem \eqref{ev-u} with $\bar u$ and integrating over $\Omega$ we get
\begin{align*}
0&=\lambda |u|_\Omega^2 -({\rm div}\; T|u)_\Omega
=\lambda |u|_\Omega^2 +\int_\Omega T:\nabla \bar{u}\,dx +\int_{\Gamma_*} (T_2\nu\bar{u}_2-T_1\nu\bar{u}_1)\,ds\\
&=\lambda |u|_\Omega^2 + 2|\mu_*^{1/2}D|_\Omega^2 +([\![T\nu]\!]|u)_{\Gamma_*}\\
&=\lambda |u|_\Omega^2 + 2|\mu_*^{1/2}D|_\Omega^2 +\sigma\bar{\lambda} (\cA_* h|h)_{\Gamma_*} +\sigma(\cA_*h|j)_{\Gamma_*},
\end{align*}
since $[\![u]\!]=0$, $[\![T\nu]\!]=\sigma (\cA_*h)\nu$ and $u\cdot\nu = \lambda h+j$ with
$(l_*/\theta_*)j=-[\![d_*\partial_\nu\vartheta]\!]$.
On the other hand, the inner product of the equation for $\vartheta$ with  $\bar\vartheta $ and an integration by parts leads to
\begin{align*}
0&= \lambda|\kappa_*^{1/2}\vartheta|_\Omega^2 + |d_*^{1/2}\nabla\vartheta|_\Omega^2 +([\![d_*\partial_\nu\vartheta]\!]|\vartheta)_{\Gamma_*}\\
&= \lambda|\kappa_*^{1/2}\vartheta|_\Omega^2 + |d_*^{1/2}\nabla\vartheta|_\Omega^2
-\sigma(j|\cA_*h)_{\Gamma_*}/\theta_*,
\end{align*}
where we employed the relations $(l_*/\theta_*)j=-[\![d_*\partial_\nu\vartheta]\!]$ and $l_*\vartheta=\sigma \cA_*h$.
Adding these identities and taking real parts  yields the important relation
\begin{equation}
\label{evid}
\begin{split}
0&={\rm Re}\,\lambda |u|_\Omega^2 + 2|\mu_*^{1/2}D|_\Omega^2
+\sigma\,{\rm Re}\,\lambda\, (\cA_* h|h)_{\Gamma_*} \\
&\quad +\theta_*\Big({\rm Re}\,\lambda|\kappa_*^{1/2}\vartheta|_\Omega^2
+ |d_*^{1/2}\nabla\vartheta|_\Omega^2\Big).
\end{split}
\end{equation}
On the other hand, if $\beta:={\rm Im}\, \lambda\neq0$, then taking imaginary parts separately we get with $a=\sigma(\cA_*h|j)_{\Gamma_*}$
\begin{equation*}
\begin{split}
0&=\beta|u|_\Omega^2-\sigma\beta (\cA_* h|h)_{\Gamma_*} + {\rm Im}\,a, \\
0&=\beta \theta_*|\kappa_*^{1/2}\vartheta|_\Omega^2 +{\rm Im}\, a.
\end{split}
\end{equation*}
Hence
$$ \sigma (\cA_* h|h)_{\Gamma_*} = |u|_\Omega^2-\theta_*|\kappa_*^{1/2}\vartheta|_\Omega^2.$$
Inserting this identity into \eqref{evid} leads to
$$0=2{\rm Re}\,\lambda |u|_\Omega^2 + 2|\mu_*^{1/2}D|_\Omega^2
+ \theta_*|d_*^{1/2}\nabla\theta|_\Omega^2,$$
which by \eqref{ev-u}-\eqref{ev-theta} (and Korn's inequality in case ${\rm Re}\,\lambda =0$)
 shows that if  $\lambda$ is an eigenvalue of $-L$ with ${\rm Re}\, \lambda\geq 0$ then $\lambda$ is real.

Supposing that $\lambda>0$ is an eigenvalue, we decompose  $\vartheta=\vartheta_0 +\bar{\vartheta}$, $h=h_0+\bar{h}$ and $j=j_0+\bar{j}$, where
$$ \bar{\vartheta}= (\kappa_*|\vartheta)_\Omega/(\kappa_*|1)_\Omega,
\quad \bar{h}=(h|1)_{\Gamma_*}/|{\Gamma_*}|,\quad \bar{j}=(j|1)_{\Gamma_*}/|{\Gamma_*}|$$
are weighted means. Then
\begin{equation*}
\begin{split}
|\kappa_*^{1/2}\vartheta|_\Omega^2 = |\kappa_*^{1/2}\vartheta_0|_\Omega^2\! +(\kappa_*|1)_\Omega \bar{\vartheta}^2,\;\;
|h|_{\Gamma_*}^2 =|h_0|_{\Gamma_*}^2 \!+ |{\Gamma_*}|\, \bar{h}^2,\;
|j|_{\Gamma_*}^2 =|j_0|_{\Gamma_*}^2 \!+ |{\Gamma_*}|\, \bar{j}^2.
\end{split}
\end{equation*}
Therefore \eqref{evid} becomes
\begin{equation}
\label{evid1}
\begin{split}
0&=\lambda |u|_\Omega^2 + 2|\mu_*^{1/2}D|_\Omega^2 +\sigma\lambda (\cA_* h_0|h_0)_{\Gamma_*}\\
&\quad
+\theta_*\Big(\lambda|\kappa_*^{1/2}\vartheta_0|_\Omega^2 + |d_*^{1/2}\nabla\vartheta_0|_\Omega^2\Big)
+ \lambda\theta_*(\kappa_*|1)_\Omega\bar{\vartheta}^2-\lambda\sigma\frac{n-1}{R_*^2}|{\Gamma_*}|\bar{h}^2.
\end{split}
\end{equation}
We further have
$$\lambda \int_{{\Gamma_*}} h\, ds = \int_{\Gamma_*} (u\cdot\nu -j)\,ds= -\int_{\Gamma_*} j\, ds$$
hence $\lambda \bar{h}=-\bar{j}$.
Also, the identity
$$(l_*/\theta_*)\int_{\Gamma_*} j\,ds = -\int_{\Gamma_*} [\![d_*\partial_\nu\vartheta]\!]\,ds
= \int_\Omega d_*\Delta \vartheta\,dx = \lambda\int_\Omega\kappa_*\vartheta\, dx$$
implies
$(l_*/\theta_*)|{\Gamma_*}| \bar{j}=\lambda(\kappa_*|1)_\Omega\bar{\vartheta}.$
Thus  \eqref{evid1} becomes
\begin{equation}
\label{evid2}
\begin{split}
0&=\lambda |u|_\Omega^2 + 2|\mu_*^{1/2}D|_\Omega^2 +\sigma\lambda (\cA_* h_0|h_0)_{\Gamma_*}\\
&\quad
+\theta_*\Big(\lambda|\kappa_*^{1/2}\vartheta_0|_\Omega^2 + |d_*^{1/2}\nabla\vartheta_0|_\Omega^2\Big)
+ \lambda|{\Gamma_*}|\Big\{\frac{l_*^2|{\Gamma_*}|}{\theta_*(\kappa_*|1)_\Omega}-\frac{\sigma(n-1)}{ R_*^2}\Big\}\bar{h}^2.
\end{split}
\end{equation}
As $\cA_*$ is positive semidefinite on functions with mean zero
if $\Gamma_*$ is connected, in this case $-L$ has no positive eigenvalues if the {\em stability condition}
\begin{equation}
\label{sc}
\frac{l_*^2|{\Gamma_*}|}{\theta_*(\kappa_*|1)_2 } - \frac{(n-1)\sigma}{ R_*^2}\ge 0
\end{equation}
is satisfied. This is the same condition we found for the thermodynamically consistent Stefan problem with surface tension; see \cite{PSZ10} and \cite{PrSi08}.

\medskip
\noindent
{\bf (ii)} The assertion follows immediately from the results in \cite[Section 3]{PSSS11}.

\medskip
\noindent
{\bf (iii)} On the other hand, if the stability condition does not hold or if ${\Gamma_*}$ is disconnected, then there is always a positive eigenvalue. To prove this we proceed as follows.
Solve the Stokes problem
\begin{equation}
\label{NDStokes}
\left\{
\begin{aligned}
\lambda u-\mu_* \Delta u +\nabla\pi &=0 &&\text{in} && \Omega\setminus{\Gamma_*}\\
{\rm div}\, u  & =0 &&\text{in} && \Omega\setminus{\Gamma_*}\\
u &=0 &&\text{on} && \partial\Omega \\
[\![u]\!] &= 0 &&\text{on} && {\Gamma_*}\\
 -[\![T\nu]\!] &= g\nu &&\text{on} && {\Gamma_*}\\
\end{aligned}
\right.
\end{equation}
and define the Neumann-to-Dirichlet operator $N_\lambda^S$ for the Stokes problem
by $N_\lambda^Sg:=u\cdot\nu.$
Similarly, solve the heat problem
\begin{equation}
\label{NDdiffusion}
\left\{
\begin{aligned}
\kappa_*\lambda\vartheta -d_*\Delta \vartheta &=0 &&\text{in}&& \Omega\setminus{\Gamma_*}\\
\partial_\nu\vartheta &=0 &&\text{on} &&\partial\Omega\\
[\![\vartheta]\!]&=0 &&\text{on} && {\Gamma_*}\\
-[\![d_*\partial_\nu\vartheta]\!] &= g &&\text{on} && {\Gamma_*}\\
\end{aligned}
\right.
\end{equation}
to obtain $\vartheta =N_\lambda^H g$,
where $N_\lambda^H$ denotes the Neumann-to Dirichlet operator for the heat problem.
In the following, we use the same notation for $\vartheta$ and its restriction to $\Gamma_*$.
\medskip

Suppose that $\lambda>0$ is an eigenvalue with eigenfunction $(u,\vartheta,h)$.
Choosing $g=-\sigma \cA_*h$ in \eqref{NDStokes} we obtain
$u\cdot\nu=-N_\lambda^S \sigma \cA_*h.$
Next we solve the heat problem \eqref{NDdiffusion} with $g=(l_*/\theta_*)(u\cdot\nu-\lambda h)$,
yielding
$$\vartheta= -(l_*/\theta_*)N_\lambda^H\big(N_\lambda^S\sigma\cA_\ast h +\lambda h).$$
This implies with the linearized Gibbs-Thomson law $l_*\vartheta=\sigma\cA_*h$
the relationship
$$-(l_*^2/\theta_*)N_\lambda^H\big(N_\lambda^S\sigma\cA_\ast h +\lambda h\big)= \sigma \cA_* h,$$
hence
$$\lambda h +  [N_\lambda^S + ((l_*^2/\theta_*)N_\lambda^H)^{-1}]\sigma A_* h =0.$$
Setting
$$ T_\lambda :=[N_\lambda^S + ((l_*^2/\theta_*)N_\lambda^H)^{-1}]^{-1}$$
we arrive at the equation
\begin{equation}
\label{Tlambda}
B_\lambda h:=\lambda T_\lambda h +\sigma \cA_* h =0.
\end{equation}
$\lambda>0$ is an eigenvalue of $-L$ if and only if \eqref{Tlambda} admits a nontrivial solution. We consider this problem in $L_2({\Gamma_*})$.
Then $\cA_*$ is selfadjoint and
$$\sigma(\cA_* h|h)_{\Gamma_*} \geq -\frac{\sigma(n-1)}{R_*^2} |h|^2_{\Gamma_*}.$$
On the other hand, we will see below that $N_\lambda^H$ and $N_\lambda^S$ are selfadjoint and positive semidefinite on $L_2({\Gamma_*})$ and hence $T_\lambda$ is selfadjoint and positive semidefinite as well. Moreover, since $\cA_*$ has compact resolvent, the operator $B_\lambda$ has compact resolvent as well, for each $\lambda>0$. Therefore the spectrum of $B_\lambda$ consists only of eigenvalues which, in addition, are real. We intend to prove that in case either $\Gamma_*$ is disconnected or the stability condition does not hold, $B_{\lambda_0}$ has $0$ as an eigenvalue, for some $\lambda_0>0$.

To proceed we need properties of the relevant Neumann-to-Dirichlet operators.
\begin{proposition}
\label{NSlambda}
The Neumann-to-Dirichlet operator $N_\lambda^S$ for the Stokes problem \eqref{NDStokes} has the following properties in $L_2(\Gamma_*)$.\\
{\bf (i)} \, If $u$ denotes the solution of \eqref{NDStokes}, then
$$ (N_\lambda^Sg|g)_{\Gamma_*} = \lambda |u|^2_{\Omega} + 2\int_\Omega \mu_*|D|_2^2\, dx,
\quad  g\in L_2(\Gamma_*),\;\lambda\geq0.$$
{\bf (ii)} \, For each $\alpha\in(0,1/2)$ there is a constant $C>0$ such that
$$(N_\lambda^S g|g)_{\Gamma_*}\geq \frac{(1+\lambda)^\alpha}{C}|N_\lambda^Sg|_{\Gamma_*}^2,\quad g\in L_2(\Gamma_*),\; \lambda\geq0.$$
In particular,
$$ |N_\lambda^S|_{\cB(L_2(\Gamma_*))}\leq \frac{C}{(1+\lambda)^\alpha}, \quad \lambda\geq0.$$
{\bf (iii)} \, Let $\Gamma_*^k$ denote the components of $\Gamma_*$ and let ${\sf e}_k$ be the function which is one on $\Gamma_*^k$, zero elsewhere. Then $(N_\lambda^Sg|{\sf e}_k)_{\Gamma_*}=0$ for each $k$; in particular $N_\lambda^Sg$ has mean value zero for each $g\in L_2(\Gamma_*)$.
Moreover, with ${\sf e}=\sum_k {\sf e}_k$ we have $N_\lambda^S{\sf e}=0$ and $(N_\lambda^Sg|{\sf e})_{L_2(\Gamma_*)}=0$ for all $g\in L_2(\Gamma_*)$.
\end{proposition}
\begin{proof}
The first assertion follows from the divergence theorem. The second assertion is a consequence of
trace and interpolation theory, combined with Korn's inequality.
The last assertion is implied with ${\rm div}\, u=0$ by the divergence theorem.
\end{proof}

\begin{proposition}
\label{NHlambda}
The Neumann-to-Dirichlet operator $N_\lambda^H$ for the diffusion problem \eqref{NDdiffusion} has the following properties in $L_2(\Gamma_*)$.\\
{\bf (i)} \,  If $\vartheta$ denotes the solution of \eqref{NDdiffusion}, then
$$ (N_\lambda^Hg|g)_{\Gamma_*} = \lambda |\sqrt{\kappa_*}\vartheta|^2_{\Omega}
+ |\sqrt{d_*}\nabla\vartheta|_{\Omega}^2 ,\quad g\in L_2(\Gamma_*),\;\lambda > 0. $$
{\bf (ii)} \, For each $\alpha\in(0,1/2)$ and $\lambda_0>0$ there is a constant $C>0$ such that
$$(N_\lambda^H g|g)_{\Gamma_*}\geq \frac{\lambda^\alpha}{C}|N_\lambda^Hg|_{\Gamma_*}^2,\quad g\in L_2(\Gamma_*),\; \lambda\geq\lambda_0.$$
In particular, $N_\lambda^H$ is injective, and
$$ |N_\lambda^H|_{\cB(L_2(\Gamma_*))}\leq \frac{C}{\lambda^\alpha}, \quad \lambda\geq\lambda_0.$$
{\bf (iii)} \, On $L_{2,0}(\Gamma_*)= \{g\in L_2(\Gamma_*):\, (g|{\sf e})_{\Gamma_*}=0\}$, we even have
$$(N_\lambda^H g|g)_{\Gamma_*}\geq \frac{(1+\lambda)^\alpha}{C}|N_\lambda^Hg|_{\Gamma_*}^2,\quad g\in L_{2,0}(\Gamma_*),\; \lambda>0,$$
and
$$ |N_\lambda^H|_{\cB(L_{2,0}(\Gamma_*))}\leq \frac{C}{(1+\lambda)^\alpha}, \quad \lambda>0.$$
In particular, for $\lambda=0$, \eqref{NDdiffusion} is solvable if and only if $(g|{\sf e})_{\Gamma_*}=0$, and then the solution is unique up to a constant.
\end{proposition}
\begin{proof}
The first assertion follows from the divergence theorem. The second and third assertions  are consequences of trace and interpolation theory, combined with Poincar\'e's inequality. The last assertion is a standard statement in the theory of elliptic transmission problems.
\end{proof}

\medskip

\noindent
(a) \, Consider $v_\lambda:=T_\lambda {\sf e}$, or equivalently $ {\sf e}=N_\lambda^S v_\lambda +
(c_* N_\lambda^H)^{-1}v_\lambda$, where we used the abbreviation $c_*=l_*^2/\theta_*$,
and where ${\sf e}$ is the characteristic function on $\Gamma_*$.
Here $\Gamma_*$ can be either connected or disconnected.
Denoting the orthogonal projection from $L_2(\Gamma_*)$ to $L_{2,0}(\Gamma_*)$ by $Q_0$, the equation for $v_\lambda$ is equivalent to
$$ v_\lambda + c_*N_\lambda^HQ_0 N_\lambda^S v_\lambda =c_*N_\lambda^H{\sf e},$$
due to Proposition \ref{NSlambda}. Multiplying this identity in $L_2(\Gamma_*)$ by $N_\lambda^Sv_\lambda$ we obtain with Propositions \ref{NSlambda} and \ref{NHlambda}
\begin{align*}
c(\lambda) |N_\lambda^Sv_\lambda|_{\Gamma_*}^2 &\leq (v_\lambda
+ c_*N_\lambda^HN_\lambda^S v_\lambda|N_\lambda^Sv_\lambda)_{\Gamma_*}
=(c_*N^H_\lambda {\sf e}|N_\lambda^Sv_\lambda)_{\Gamma_*}\\
&= c_*({\sf e}|N_\lambda^HQ_0N_\lambda^Sv_\lambda)_{\Gamma_*}
\leq C(\lambda)|N_\lambda^Sv_\lambda|_{\Gamma_*},
\end{align*}
where $c(\lambda)$ and $C(\lambda)$ are bounded near $\lambda=0$,
showing that $N_\lambda^Sv_\lambda$ is bounded near $\lambda=0$. This implies
$$\lim_{\lambda\to0} \lambda T_\lambda{\sf e}=\lim_{\lambda\to0} \lambda v_\lambda = c_*\lim_{\lambda\to0} \lambda N^H_\lambda{\sf e},$$
provided the latter limit exists.

To compute this limit, we proceed as follows. First we solve the problem
\begin{equation}
\label{NDdiffusion0}
\left\{
\begin{aligned}
 -d_*\Delta \vartheta &= -\kappa_* a_0 &&\text{in} && \Omega\setminus{\Gamma_*}\\
\partial_\nu\vartheta & =0 &&\text{on} && \partial\Omega \\
[\![\vartheta]\!] &=0 &&\text{on} && {\Gamma_*} \\
 -[\![d_*\partial_\nu\vartheta]\!] &= {\sf e} &&\text{on} && {\Gamma_*},\\
\end{aligned}
\right.
\end{equation}
where $a_0=|\Gamma_*|/(\kappa_*|1)_\Omega$,
which is solvable since the necessary compatibility condition holds. We denote the solution by $\vartheta_0$ and normalize it by $(\kappa_*|\vartheta_0)_\Omega=0$.
Then $\vartheta_\lambda=N_\lambda^H{\sf e}-\vartheta_0- a_0/\lambda$ solves the problem
\begin{equation}
\label{NDdiffusion1}
\left\{
\begin{aligned}
\kappa_* \lambda\vartheta -d_*\Delta \vartheta &=-\kappa_*\lambda \vartheta_0
&&\text{in} && \Omega\setminus{\Gamma_*} \\
\partial_\nu\vartheta &=0 &&\text{on} && \partial\Omega\\
[\![\vartheta]\!] &=0, &&\text{on} && {\Gamma_*} \\
 -[\![d_*\partial_\nu\vartheta]\!] &= 0 &&\text{on} &&{\Gamma_*}.\\
\end{aligned}
\right.
\end{equation}
By the normalization $(\kappa_*|\vartheta_0)_\Omega=0$  we see that $\vartheta_\lambda$ is bounded in $H^2_2(\Omega\setminus\Gamma_*)$ as $\lambda\to0$. Hence we have
$$\lim_{\lambda\to0}\lambda N_\lambda^H{\sf e} = \lim_{\lambda\to0}[\lambda \vartheta_\lambda+ \lambda \vartheta_0+ a_0]= a_0 =|\Gamma_*|/(\kappa_*|1)_\Omega.$$
This then implies
$$\lim_{\lambda\to 0} (B_\lambda {\sf e}|{\sf e})_{\Gamma_*}
= c_* \frac{|\Gamma_*|^2}{(\kappa_*|1)_\Omega}- \sigma|\Gamma_*|\frac{(n-1)}{R^2_*} <0,$$
if the stability condition does not hold.

\medskip

\noindent
(b) \, Next suppose that $\Gamma_*$ is disconnected,
i.e. $\Gamma_*=\cup_{k=1}^m\Gamma^k_*$,
and set $g=\sum_k a_k {\sf e}_k\neq0$ with $\sum_k a_k=0$. Hence $Q_0g=g$. Then for $v_\lambda:=T_\lambda {\sf e}$ we have as in (a) boundedness of $N_\lambda^Sv_\lambda$ and then
$$\lim_{\lambda\to0} \lambda T_\lambda g=\lim_{\lambda\to0} \lambda v_\lambda = c_*\lim_{\lambda\to0} \lambda N^H_\lambda Q_0g=0,$$
since $N_\lambda^HQ_0$ is bounded as $\lambda\to0$. This implies
$$\lim_{\lambda\to0} (B_\lambda g|g)_{L_2(\Gamma_*)} = - \frac{\sigma(n-1)}{R_*^2} \sum_k |\Gamma_*^k|a_k^2<0.$$

\medskip

\noindent
(c) \, Next we consider the behavior of $(B_\lambda g|g)_{L_2(\Gamma_*)}$ as $\lambda\to\infty$. With $c_*=l_*^2/\theta_*$ as above we first have
$$ T_\lambda = (I+c_*N_\lambda^HN_\lambda^S)^{-1}c_*N_\lambda^H = c_*N_\lambda^H - c_*N_\lambda^H N_\lambda^S(I+c_*N_\lambda^HN_\lambda^S)^{-1}c_*N_\lambda^H,$$
hence by Propositions \ref{NHlambda}, \ref{NSlambda} for $\lambda\geq\lambda_0$,
with $\lambda_0$ sufficiently large,
\begin{align*}
(T_\lambda g|g)_{\Gamma_*} &= c_*(N^H_\lambda g|g)_{\Gamma_*}-c_*^2( N_\lambda^S(I+c_*N_\lambda^HN_\lambda^S)^{-1}N_\lambda^Hg |N_\lambda^Hg)_{\Gamma_*}\\
&\geq c_*\Big[(N^H_\lambda g|g)_{\Gamma_*}- c_*^2\frac{|N_\lambda^S|_{L_2(\Gamma_*)}}{1-c_*|N_\lambda^H|_{L_2(\Gamma_*)}|N_\lambda^S|_{L_2(\Gamma_*)}}|N_\lambda^Hg|_{\Gamma_*}^2\Big]\\
&\geq c_*\Big[(N^H_\lambda g|g)_{\Gamma_*}-\frac{C\lambda_0^{-\alpha}|N_\lambda^S|_{L_2(\Gamma_*)}}{1-c_*|N_\lambda^H|_{L_2(\Gamma_*)}|N_\lambda^S|_{L_2(\Gamma_*)}}(N_\lambda^Hg|g)_{\Gamma_*}\Big]\\
&\geq c_*\Big[(N^H_\lambda g|g)_{\Gamma_*}- \frac{1}{2}(N^H_\lambda g|g)_{\Gamma_*}\Big]
=c_0(N^H_\lambda g|g)_{\Gamma_*}.
\end{align*}
Therefore, it is sufficient to bound $(N^H_\lambda g|g)_{\Gamma_*}$ from below as $\lambda\to \infty$.

For this purpose we introduce the projections $P$ and $Q$ by
$$Pg = c_m\sum_{k=1}^m (g|{\sf e}_k)_{\Gamma_*}{\sf e}_k,\quad Q=I-P,$$
where $ c_m=m/|\Gamma_*|$ in case $\Gamma_*$ has $m$ components.
Then with $g_k =(g|{\sf e}_k)_{\Gamma_*}$
\begin{align*}
|(N^H_\lambda Pg|Qg)_{\Gamma_*}|&\leq c_m \sum_k |g_k|\,|(N_\lambda^HQg|{\sf e}_k))_{\Gamma_*}|\\
&\leq C\sum_k |g_k|\,|N_\lambda^HQg|_{\Gamma_*}
\leq C\lambda^{-\alpha/2}\sum_k|g_k|(N_\lambda^HQg|Qg)^{1/2}_{\Gamma_*}\\
&\leq C\lambda^{-\alpha/2}\Big[\sum_k|g_k|^2 +m(N_\lambda^HQg|Qg)_{\Gamma_*}\Big]\\
&\leq C\lambda^{-\alpha/2}\Big[|Pg|_{\Gamma_*}^2 +(N_\lambda^HQg|Qg)_{\Gamma_*}\Big],
\end{align*}
where $C>0$ is a generic constant, which may differ from line to line.
Hence for $\lambda\geq\lambda_0$, with $\lambda_0$ sufficiently large, we have
\begin{align*}
(N_\lambda^Hg|g)_{\Gamma_*}&= (N_\lambda^HQg|Qg)_{\Gamma_*}+2(N_\lambda^HQg|Pg)_{\Gamma_*}+ (N_\lambda^HPg|Pg)_{\Gamma_*}\\
&\geq \frac{1}{2}(N_\lambda^HQg|Qg)_{\Gamma_*}+ (N_\lambda^HPg|Pg)_{\Gamma_*}- \frac{C}{\lambda_0^{\alpha/2}}|Pg|_{\Gamma_*}^2.
\end{align*}
This implies
\begin{align*}
(B_\lambda g|g)_{\Gamma_*}&=\lambda(T_\lambda g|g)_{\Gamma_*} +\sigma(\cA_*g|g)_{\Gamma_*}\\
&\geq c_0\big[\frac{\lambda}{2}(N_\lambda^HQg|Qg)_{\Gamma_*}+ \lambda(N_\lambda^HPg|Pg)_{\Gamma_*}\big]\\
&\quad +c_0\sigma (\cA_*Qg|Qg)_{\Gamma_*}- c|Pg|_{\Gamma_*}^2.
\end{align*}
Since $N_\lambda^H$ is positive semidefinite
and also $\cA_*Q$ has this property
as ${\rm im}\,(Q)\subset L_{2,0}(\Gamma_*)$,
we only need to prove that $\lambda(N_\lambda^HPg|Pg)_{\Gamma_*}$ tends to infinity as $\lambda\to\infty$.

To prove this, similarly as before we assume $\lambda\ge \lambda_0$ and estimate
\begin{equation*}
 |(N_\lambda^H{\sf e}_i|{\sf e}_j)_{L_2(\Gamma_*)}|\leq C|N_\lambda^H{\sf e}_i|_{L_2(\Gamma_*)}
\leq C\lambda_0^{-\alpha/2}(N_\lambda^H{\sf e}_i|{\sf e}_i)^{1/2}_{L_2(\Gamma_*)}.
\end{equation*}
Choosing $\lambda_0$ sufficiently large this yields
\begin{equation*}
(N_\lambda^HPg|Pg)_{L_2(\Gamma_*)}
\geq c_0\Big[ \min_i (N_\lambda^H{\sf e}_i|{\sf e}_i)_{L_2(\Gamma_*)}
- \frac{C}{\lambda_0^{\alpha}}\Big]
|Pg|_{L_2(\Gamma_*)}^2.
\end{equation*}
Therefore it is sufficient to show
\begin{equation}
\label{asymptNH}
\lim_{\lambda\to\infty} \lambda(N_\lambda^H{\sf e}_k|{\sf e}_k)_{L_2(\Gamma_*)}=\infty,
\quad k=1,\ldots,m.
\end{equation}
So suppose, on the contrary, that $\lambda_j(N_{\lambda_j}^Hg|g)_{L_2(\Gamma_*)}$ is bounded, for some $g={\sf e}_k$ and some sequence $\lambda_j\to\infty$. Then the corresponding solution $\vartheta_j$ of \eqref{NDdiffusion} is such that $v_j:=\lambda_j\vartheta_j$ is bounded in $L_2(\Omega)$ as
\begin{equation*}
\lambda_j^2 |\sqrt{\kappa_*}\vartheta_j|^2_{\Omega}
\le \lambda_j\big(\lambda_j |\sqrt{\kappa_*}\vartheta_j|^2_{\Omega}
+ |\sqrt{d_*}\nabla\vartheta_j|_{\Gamma_*}^2 \big)
=\lambda_j(N_{\lambda_j}^Hg|g)_{\Gamma_*}.
\end{equation*}
Hence $v_j$ has a weakly convergent subsequence, and we can assume without loss of generality
that  $v_j\to v_\infty$ weakly in $L_2(\Omega)$. Fix a test function $\psi\in \cD(\Omega\setminus\Gamma_*)$. Then
$$(\kappa_* v_j|\psi)_\Omega = (d_*\Delta \vartheta_j|\psi)_\Omega=
(\vartheta_j|d_*\Delta\psi)_\Omega=(v_j|d_*\Delta\psi)_\Omega/\lambda_j\to 0$$
as $j\to\infty$, hence $v_\infty=0$ in $L_2(\Omega)$. On the other hand we have
\begin{align*}
0<\frac{|\Gamma_*|}{m} &= \int_{\Gamma_*} g\,ds =
\int_{\Gamma_*} -[\![d_*\partial_\nu \vartheta_j]\!]\,ds\\
&=\int_\Omega d_*\Delta\vartheta_j\, dx = \lambda_j\int_\Omega\kappa_*\vartheta_j\, dx \to\int_\Omega \kappa_* v_\infty dx,
\end{align*}
hence $v_\infty$ is nontrivial, a contradiction. This implies that \eqref{asymptNH} is valid.

\medskip

\noindent
(d) \, Summarizing, we have shown that $B_\lambda$ is not positive semidefinite for small $\lambda>0$ if either $\Gamma_*$ is not connected or the stability condition does not hold, and $B_\lambda$ is always positive semidefinite for large $\lambda$. Set
$$\lambda_0 = \sup\{\lambda>0:\, B_\mu \mbox{ is not positive semidefinite for each } \mu\in(0,\lambda]\}.$$
Since $B_\lambda$ has compact resolvent, $B_\lambda$ has a negative eigenvalue for each $\lambda<\lambda_0$. This implies that $0$ is an eigenvalue of
$B_{\lambda_0}$, thereby proving that $-L$ admits the positive eigenvalue $\lambda_0$.

Moreover, we have also shown that
$$ B_0h:= \lim_{\lambda\to0} \lambda T_\lambda h + \sigma \cA_* h
= c_*\frac{|\Gamma_*|}{(\kappa_*|1)_\Omega}(I-Q_0)h + \sigma \cA_* h .$$
Therefore, $B_0$ has the eigenvalue
$c_*|\Gamma_*|/(\kappa_*|1)_{L_2(\Omega)}-\sigma(n-1)/R_*^2$ with
eigenfunction ${\sf e}$, and in case $m>1$ it also possesses the
eigenvalue $-\sigma(n-1)/R_*^2$ with precisely $(m-1)$ linearly
independent eigenfunctions of the form $\sum_k a_k {\sf e}_k$ with
$\sum_k a_k=0$. This implies that $-L$ has exactly $m$
positive eigenvalues if the stability condition does not hold, and
$m-1$ otherwise.

\medskip

\noindent
{\bf (iv)} (a) Suppose that $(u,\vartheta,h)$ is an eigenfunction of
$L$ for the eigenvalue $\lambda=0$.
Then \eqref{evid} yields
\begin{equation}
\label{lambda=0}
2|\mu_*^{1/2}D|_\Omega^2+|d_*^{1/2}\nabla\vartheta|_\Omega^2=0.
\end{equation}
It follows from \eqref{lambda=0} and
\eqref{ev-u}-\eqref{ev-theta} that
$\vartheta$ is constant and $D=0$ on $\Omega$.
Korn's inequality, in turn, implies $\nabla u=0$ on $\Omega$, and we then have $u=0$
by the no-slip condition on $\partial \Omega$.
Moreover, the pressures are constant in the phases and we have
\begin{equation*}
[\![\pi]\!]+\sigma \cA_* h=0, \quad l_*\vartheta -\sigma \cA_*h=0\quad\text{on}
\quad \Gamma_\ast.
\end{equation*}
We can now conclude from
the relation $l_*\vartheta -\sigma \cA_*h=0$ that
the kernel of $L$ is given by
\begin{equation}
\label{N-L-m}
N(L) = {\rm span}\big\{(0,\frac{-\sigma(n-1)}{l_*R^2_*},{\sf e}),
(0,0,Y^k_1),\ldots,(0,0,Y^k_n) : \, 1\le k\le m\big\},
\end{equation}
where the functions $Y^k_j=Y^k_j{\sf e}_k$
denote the {\em spherical harmonics of degree one} on $\Gamma_\ast^k$,
normalized by $(Y^k_i|Y^k_j)_{\Gamma^k_*}=\delta_{ij}$.
This shows that $N(L)$ has dimension $(mn+1)$,
in accordance with the situation for the Stefan problem with surface tension \cite{PSZ10}.
\medskip

\noindent
(b) It remains to show that $\lambda=0$ is semi-simple if $s\neq 0 $.
We concentrate on the case where $\Gamma_*$ is connected,
for simplicity. The disconnected case is treated in complete analogy.
So suppose $(u,\vartheta,h)\in N(L^2)$. Hence $L(u,\vartheta,h)\in N(L)$, i.e.
\begin{equation*}
L(u,\vartheta,h)=\alpha_0(0,-{\sigma (n-1)}/(l_*R^2_*),Y_0)+\sum_{l=1}^n \alpha_l (0,0, Y_l),
\end{equation*}
where  $\alpha_0,\alpha_l$ are appropriate coefficients and $Y_0=1$.
Thus $(u,\vartheta,h)$ solves the equations
\begin{equation}
\label{semi-simple-u}
\left\{
\begin{aligned}
-\mu_* \Delta u +\nabla\pi & =0 &&\text{in} && \Omega\setminus{\Gamma_*}\\
{\rm div}\, u &=0 &&\text{in} && \Omega\setminus{\Gamma_*}\\
u &= 0 &&\text{on} && \partial\Omega \\
[\![u]\!] &= 0  &&\text{on} &&{\Gamma_*}\\
 -[\![T\nu]\!] &= -\sigma (\cA_* h)\nu &&\text{on} && {\Gamma_*},\\
\end{aligned}
\right.
\end{equation}
and
\begin{equation}
\label{semi-simple-theta}
\left\{
\begin{aligned}
 -d_*\Delta \vartheta &=-\alpha_0\kappa_*\sigma(n-1)/l_*R^2_*
 &&\text{in} && \Omega\setminus{\Gamma_*}\\
 \partial_\nu\vartheta &=0 &&\text{on} && \partial\Omega \\
[\![\vartheta]\!] &=0 &&\text{on} &&{\Gamma_*}\\
l_*\vartheta -\sigma \cA_*h &= 0  &&\text{on} && {\Gamma_*}\\
-(l_*/\theta_*)u\cdot\nu -[\![d_*\partial_\nu\vartheta]\!] &=
 (l_*/\theta_*) \Sigma_{l=0}^n{\alpha_l}Y_l
&&\text{on} &&{\Gamma_*},\\
\end{aligned}
\right.
\end{equation}
We have to show $\alpha_l=0$ for all $l$.
Integrating the equation for the temperature over $\Omega$ we find
\begin{equation}
\label{alpha-0}
\alpha_0 \frac{\sigma(n-1)(\kappa_*|1)_\Omega} {l_*R^2_*}
=\alpha_0 \frac{ l_*|\Gamma_*|}{\theta_*},
\end{equation}
as $u\cdot \nu$ and the spherical harmonics $Y_l,\; 1\le l\le n$, all have mean zero on $\Gamma_*$.
Therefore, $\alpha_0=0$, unless there is equality in the stability condition.
If $s\neq 0$, and hence  $\alpha_0=0,$ it follows from \eqref{semi-simple-u} and \eqref{semi-simple-theta}
\begin{equation*}
\begin{split}
0&=2|\mu_*^{1/2}D|_\Omega^2 +\sigma(\cA_*h|u\cdot \nu)_{\Gamma_*},\\
0&= \theta_*|d_*^{1/2}\nabla\vartheta|_\Omega^2-
\sigma(\cA_*h|u\cdot\nu +\sum_{l=1}^n\alpha_l Y_l)_{\Gamma_*}
=\theta_*|d_*^{1/2}\nabla\vartheta|_\Omega^2
- \sigma(\cA_*h|u\cdot\nu )_{\Gamma_*},
\end{split}
\end{equation*}
as $\cA_*$ is self-adjoint and $\cA Y_l=0$ for the spherical harmonics.
Adding these equations gives
\begin{equation*}
2|\mu_*^{1/2}D|_\Omega^2 + \theta_*|d_*^{1/2}\nabla\vartheta|_\Omega^2=0.
\end{equation*}
This implies $D=0$, $\vartheta$ constant, $u\cdot \nu=0$  and $u=0$, which in turn yields
$ 0=\sum_{l=1}^n{\alpha_l}Y_l.$
Thus $\alpha_l=0$ for all $l$ since the spherical harmonics $Y_l$ are linearly independent.
Therefore, the eigenvalue $\lambda=0$ is semi-simple.

\medskip
\noindent
{\bf (v)}
Suppose for the moment that $\Gamma_*$ consists of a single sphere of radius
$R_*=\sigma(n-1)/[\![\psi(\theta_*)]\!]$, centered at the origin of $\R^n$.
Suppose ${\mathcal S}$ is a sphere that is
sufficiently close to $\Gamma_*$. Denote by $(\zeta_1,\ldots,\zeta_n)$ the
coordinates of its center and let $\zeta_0$ be such that
$\sigma(n-1)/[\![\psi(\theta_*+\theta_*\zeta_0)]\!]$
corresponds to its radius.
We observe that the equation $\sigma(n-1)/[\![\psi(\theta_*+\theta_*\zeta_0)]\!]=R$ has a unique
solution $\zeta_0$ for $R$ close to $R_*$, as
$[\![\psi^\prime( \theta_*)]\!]\neq 0$ by assumption.
Then, by \cite[Section 6]{ES98}, the
sphere ${\mathcal S}$ can be parameterized over $\Gamma_*$ by the
distance function
\[
\rho(\zeta)=\sum_{j=1}^n \zeta_j Y_j-R_*+\sqrt{(\sum_{j=1}^n \zeta_j
Y_j)^2+(\sigma(n-1)/[\![\psi(\theta_*+\theta_*\zeta_0)]\!])^2-\sum_{j=1}^n \zeta_j^2}.
\]
Denoting by $O$ a sufficiently small neighborhood of $0$ in
$\R^{n+1}$, the mapping
\begin{equation*}
[\zeta\mapsto \Psi(\zeta):=(0,\theta_*(1+\zeta_0),\rho(\zeta))]:O\to
W^2_p(\Omega)^{n+1}\times W^{4-1/p}_p(\Gamma_*)
\end{equation*}
is $C^1$ (in fact $C^k$ if $\psi$ is $C^k$), and the derivative at $0$ is given by
\begin{equation*}
\Psi'(0)z=\big(0,\theta_*,-\sigma(n-1)\theta_* [\![\psi^\prime(\theta_*)]\!]/[\![\psi(\theta_*)]\!]^2\big)z_0
+\big(0,0,\sum_{j=1}^n z_j Y_j\big),
\quad z\in \R^{n+1}.
\end{equation*}
Noting that
$$
\frac{\sigma(n-1)\theta_* [\![\psi^\prime(\theta_*)]\!]}{[\![\psi(\theta_*)]\!]^2}
=\frac{l_*R^2_*}{\sigma(n-1)}
$$
we can conclude that
near $e_*=(0,\theta_*,\Gamma_*)$ the set $\cE$ of equilibria is a
$C^1$-manifold in $W^2_p(\Omega)^n\times W^2_p(\Omega)\times W^{4-1/p}_p(\Gamma_*)$ of
dimension $(n+1)$, and that $T_{e_*}\cE$ 
is isomorphic to the eigenspace $N(L)$.

It is now easy to see that this result remains valid
for the case of $m$ spheres of the same radius $R_\ast$. The
dimension of $\cE$ is then given by $(mn+1)$, as $mn$ parameters are
needed to locate the respective centers, and one additional
parameter is needed to track the common radius.
\end{proof}
\begin{remarks}
(a) One should observe that for the case $s=0$, the eigenvalue $\lambda=0$
ceases to be semi-simple and the dimension of the generalized eigenspace raises by one.
This can be shown by similar arguments
as in the proof of Theorem~2.1.(c) in \cite{PrSi08}.
\medskip\\
(b)
For the Fr\'echet derivative of the energy functional ${\sf E}(u,\theta,\Gamma)$,
see \eqref{energy} for the definition,
we obtain
\begin{equation*}
\langle {\sf E}^\prime(u,\theta,\Gamma)|(v,\vartheta,h)\rangle
= \int_\Omega (u\cdot v +\epsilon^\prime(\theta)\vartheta)\,dx
-\int_\Gamma (\sigma H_\Gamma +[\![\frac{1}{2}|u|_2^2 +\epsilon(\theta)]\!])h\,ds.
\end{equation*}
At equilibrium $(u,\theta,\Gamma)=(0,\theta_*,\Gamma_*)$
this yields
\begin{equation*}
\langle {\sf E}^\prime(0,\theta_*,{\Gamma_*})|(v,\vartheta,h)\rangle
= \int_\Omega\kappa_*\vartheta\, dx -\int_{\Gamma_*} (\sigma H_{\Gamma_*} +[\![\epsilon_*]\!])h\,ds.
\end{equation*}
Here
$[\![\epsilon_*]\!]:=[\![\epsilon(\theta_*)]\!]
=[\![\psi(\theta_*)]\!]-\theta_*[\![\psi^\prime(\theta_*)]\!]=-(\sigma H_{\Gamma_*}+l_\ast) $
where we used the equilibrium relation $[\![\psi(\theta_*)]\!]+\sigma H_{\Gamma_*}=0$
and the definition of $l_\ast$ in the last step.
Preservation of energy then requires
$\langle {\sf E}^\prime(0,\theta_*,{\Gamma_*})|(v,\vartheta,h)\rangle = 0$, hence
\begin{equation*}
\int_\Omega\kappa_*\vartheta\,dx + l_*\int_{\Gamma_*} h\,ds=0.
\end{equation*}
In this case  $h=\sum_{l=0}^n \alpha_l Y_l$, where $Y_0=1$
and $Y_l$ denote the orthonormalized spherical harmonics of degree one.
Hence  $\bar{h}=\alpha_0$, and
$\vartheta=-\alpha_0\sigma (n-1)\theta_*/ l_*R_*^2$, which implies
$$\alpha_0\Big[\frac{\sigma(n-1)\theta_*(\kappa_*|1)_\Omega}{l_*R^2_*} - l_*|{\Gamma_*}|\Big]=0.$$
Thus $\bar{h}=\alpha_0=0$ unless we have equality in the stability condition \eqref{sc}.
Conservation of energy kicks out one dimension of the eigenspace.
\end{remarks}

\section{Nonlinear Stability of Equilibria}
\noindent
{\bf 1.}\, We now consider problem \eqref{i2pp} in a neighborhood of a non-degenerate equilibrium $e_*=(0,\theta_*,\Gamma_*)\in\cE$, with  $l_*=l(\theta_*)\neq0$.
Setting $\Sigma=\Gamma_*$
the transformed problem becomes
\begin{equation}
\label{enonlin-u-theta}
\left\{
\begin{aligned}
\partial_t u-\mu_* \Delta u +\nabla\pi &=F_u(u,\pi,\vartheta,h) &&\text{in} &&\Omega\setminus{\Sigma}\\
{\rm div}\, u &=F_d(u,h) &&\text{in} && \Omega\setminus{\Sigma}\\
u=0,\;\;\partial_\nu\vartheta &=0 &&\text{on} &&\partial\Omega\\
[\![u]\!]=0,\;\; [\![\vartheta]\!]&= 0 &&\mbox{on} &&\Sigma\\
-P_\Sigma[\![\mu_* (\nabla u +[\nabla u]^{\sf T})]\!]\nu_\Sigma &= G_\tau(u,\vartheta,h)
 &&\mbox{on} && {\Sigma}\\
 -[\![T\nu_\Sigma\cdot\nu_\Sigma]\!] +\sigma \cA_\Sigma h &=G_\nu(u,\vartheta,h)
&&\mbox{on} && {\Sigma}\\
\kappa_*\partial_t\vartheta -d_*\Delta \vartheta &=F_\theta(u,\vartheta,h)
&&\text{in} &&\Omega\setminus{\Sigma} \\
l_*\vartheta-\sigma \cA_\Sigma h  &= G_\theta(\vartheta,h) &&\text{on} &&{\Sigma}\\
(l_*/\theta_*)(\partial_th-u\cdot\nu_\Sigma) -[\![d_*\partial_\nu\vartheta]\!]&= G_h(u,\vartheta,h)
&& \text{on} &&{\Sigma}\\
u(0)=u_0,\; \vartheta(0)=\vartheta_0,\; h(0)&=h_0. && && \\
\end{aligned}
\right.
\end{equation}
The nonlinearities on the right hand side of \eqref{enonlin-u-theta} are, up to some straightforward modifications, defined in \cite[Section 7]{PSSS11}.
It follows that the nonlinearities
are of class $C^1$ from $\EE(J)$ to $\FF(J)$, and they satisfy
\begin{equation*}
F_j(0)=G_k(0)=F_j^\prime(0)=G_k^\prime(0)=0,
\quad j\in\{u,d,\theta\},\quad  k\in\{\tau,\nu,\theta,h\}.
\end{equation*}
In order to shorten notation, we will occasionally write \eqref{enonlin-u-theta} in short form
\begin{equation*}
\LL z=N(z),\quad z(0)=z_0.
\end{equation*}
The state manifold locally near the equilibrium $e_*=(0,\theta_*,\Gamma_*)$  reads as
\begin{eqnarray}
\label{phasemanifeq}
\cSM:=&&\hspace{-0.5cm}\Big\{(u,\vartheta,h)
\in \big( W^{2-2/p}_p(\Omega\setminus\Sigma)\cap C(\bar{\Omega})\big)^{n+1}
\times W^{4-3/p}_p(\Sigma),\nonumber\\
 &&\;\;{\rm div}\, u=F_d(u,h)\; \mbox{ in }\; \Omega\setminus\Sigma,\quad u=\partial_\nu\vartheta =0 \mbox{ on } \partial\Omega,\\
 &&  -P_\Sigma[\![\mu_* (\nabla u+[\nabla u]^{\sf T}]\!]\nu_\Sigma =G_\tau(u,\vartheta,h) \mbox{ on }\; \Sigma,\nonumber\\
&& l_*\vartheta-\sigma \cA_\Sigma h=G_\theta(\vartheta,h),\quad
[\![d_*\partial_\nu \vartheta]\!] + G_h(u,\vartheta,h)\in W^{2-6/p}_p(\Sigma)\Big\}.\nonumber
\end{eqnarray}
Note that due to the compatibility conditions this is a nonlinear manifold.
We shall parameterize this manifold over its tangent space
\begin{equation*}
\label{tangphasemanifeq}
\begin{aligned}
\tilde{Z}:=&
\Big\{(\tilde{u},\tilde{\vartheta},\tilde{h})\in
 \big(W^{2-2/p}_p(\Omega\setminus\Sigma)\cap C(\bar{\Omega})\big)^{n+1}\times W^{4-3/p}_p(\Sigma),
 \\
 &{\rm div}\,\tilde{u}=0\; \mbox{ in }\; \Omega\setminus\Sigma,
 \quad \tilde{u}=\partial_\nu\tilde{\vartheta} =0 \mbox{ on } \partial\Omega,
 \\
 &-P_\Sigma[\![\mu_* (\nabla\tilde{u}+[\nabla \tilde{u}]^{\sf T}]\!]\nu_\Sigma =0,\;\;
 l_*\tilde{\vartheta}-\sigma \cA_\Sigma \tilde{h}=0,
 \;\; [\![d_*\partial_\nu \tilde{\vartheta}]\!] \in W^{2-6/p}_p(\Sigma)
\Big\}.
\end{aligned}
\end{equation*}
We mention that the norm in $\tilde{Z}$ is given by
\begin{equation*}
|(\tilde{u},\tilde{\vartheta},\tilde{h})|_{\tilde{Z}}=
|\tilde{u}|_{W^{2-2/p}_p(\Omega\setminus\Sigma)}
+ |\tilde{\vartheta}|_{W^{2-2/p}_p(\Omega\setminus\Sigma)}
+ |\tilde{h}|_{W^{4-3/p}_p(\Sigma)}
+ |[\![d_*\partial_\nu\tilde{\vartheta}]\!]|_{W^{2-6/p}_p(\Sigma)}.
\end{equation*}
\medskip
\noindent
{\bf 2.}\, In order to parameterize the state manifold $\cSM$ over $\tilde{Z}$ near
the given equilibrium $(0,\theta_*,\Sigma)$ we consider the linear elliptic problem
\begin{equation}
\label{lin-param}
\left\{
\begin{aligned}
\omega u -\mu_*\Delta u +\nabla\pi&=0 &&\text{in} && \Omega\setminus\Sigma \\
{\rm div}\, u&= f_d &&\text{in} &&\Omega\setminus\Sigma \\
u=\partial_{\nu} \vartheta &=0 &&\text{on} &&\partial \Omega\\
[\![u]\!]=[\![\vartheta]\!]&=0 &&\text{on} &&\Sigma\\
-[\![\mu_*(\nabla u +[\nabla u]^{\sf T}]\!]\nu_\Sigma +[\![\pi]\!]\nu_\Sigma
+\sigma (\cA_\Sigma h)\nu_\Sigma&=g_u &&\text{on} &&\Sigma \\
\kappa_*\omega \vartheta-d_*\Delta \vartheta&=0 &&\text{in} &&\Omega\setminus\Sigma\\
l_* \vartheta - \sigma \cA_\Sigma h &=g_\theta && \text{on} &&\Sigma\\
(l_*/\theta_*)(\omega h-u\cdot\nu_\Sigma) -[\![d_*\partial_\nu \vartheta]\!]&=g_h
&&\text{on} &&\Sigma
\end{aligned}
\right.
\end{equation}
for given data $(f_d,g_u,g_h,g_\theta)$. For this problem we have the following result.
\begin{proposition}
\label{param} Suppose $p>3$, $l_*\neq0$ and $\omega>0$ is sufficiently large. Then problem \eqref{lin-param}
admits a unique solution $(u,\pi,\vartheta,h)$ with regularity
$$ (u,\vartheta,h)\in \big( W^{2-2/p}_p(\Omega\setminus\Sigma)\cap C(\bar\Omega)\big)^{n+1}
\times W^{4-3/p}_p(\Sigma),\quad \pi\in \dot{W}^{1-2/p}_p(\Omega\setminus\Sigma),$$
if and only if the data $(f_d,g_u,g_h,g_\theta)$ satisfy
$$f_d\in W^{1-2/p}_p(\Omega\setminus\Sigma)\cap \dot{H}^{-1}_p(\Omega),  \quad
(g_u,g_h)\in W^{1-3/p}_p(\Sigma)^{n+1},\quad g_\theta\in W^{2-3/p}_p(\Sigma).$$
The solution map $[(f_d,g_u,g_h,g_\theta)\mapsto(u,\pi,\vartheta,h)]$ is continuous in the corresponding spaces.
\end{proposition}
\begin{proof}
This purely elliptic problem can be solved in the same way as the corresponding linear
parabolic problem.
\end{proof}
\begin{theorem}
\label{param-phi}
There exists a neighborhood $\tilde{U}$ of $0$ in $\tilde{Z}$ and a map
\begin{equation*}
\phi\in C^1\big(\tilde{U},(W^{2-2/p}_p(\Omega\setminus\Sigma)\cap C(\bar{\Omega}))^{n+1}
\times W^{4-3/p}_p(\Sigma)\big)\quad\text{with $\ \phi(0)=\phi^\prime(0)=0$,}
\end{equation*}
such that
$
[\tilde{z}\mapsto \tilde{z}+\phi(\tilde{z})]: \tilde{U}\to \cSM
$
provides a parameterization of the state manifold $\cSM$ near the equilibrium
$(0,\theta_*,\Sigma)$.
\end{theorem}
\begin{proof}
Fix any large $\omega>0$.
Given $\tilde{z}=(\tilde{u},\tilde{\vartheta},\tilde{h})\in \tilde{Z}$ sufficiently small, and setting
 $(u,\vartheta,h)=(\tilde{u},\tilde{\vartheta},\tilde{h})+(\bar{u},\bar{\vartheta},\bar{h})$,
 we solve the nonlinear elliptic problem
\begin{equation}
\label{param1}
\left\{
\begin{aligned}
\omega \bar{u} -\mu_*\Delta\bar{u} +\nabla\bar{\pi}&=0 &&\text{in} &&\Omega\setminus\Sigma\\
{\rm div}\, \bar{u}&= F_d(u,h) &&\text{in} && \Omega\setminus\Sigma\\
\bar{u}=\partial_{\nu} \bar{\vartheta} &=0 &&\text{on} &&\partial \Omega\\
[\![\bar{u}]\!]=[\![\bar{\vartheta}]\!]&=0 &&\text{on} &&\Sigma\\
-P_\Sigma[\![\mu_*(\nabla\bar{u} +[\nabla\bar{u}]^{\sf T}]\!]\nu_\Sigma &=G_\tau(u,\vartheta,h)
&&\text{on} &&\Sigma\\
-([\![\mu_*(\nabla\bar{u} +[\nabla\bar{u}]^{\sf T}]\!]\nu_\Sigma|\nu_\Sigma) +[\![\bar{\pi}]\!]
+\sigma \cA_\Sigma \bar{h}&= 0 &&\text{on} &&\Sigma\\
\kappa_*\omega \bar{\vartheta}-d_*\Delta \bar{\vartheta}&=0 &&\text{in} && \Omega\setminus\Sigma\\
l_* \bar{\vartheta} - \sigma A_\Sigma \bar{h} &=G_\theta(\vartheta,h) &&\text{on} &&\Sigma\\
(l_*/\theta_*)( \omega\bar{h}-\bar{u}\cdot\nu_\Sigma) -[\![d_*\partial_\nu \bar{\vartheta}]\!]
&= G_h(u,\vartheta,h)
&&\text{on} &&\Sigma
\end{aligned}
\right.
\end{equation}
by means of  the implicit function theorem, employing Proposition \ref{param}.
Then with $\bar{z}=(\bar{u},\bar{\vartheta},\bar{h})$ and $z=\tilde{z}+\bar{z}$
we obtain  $\bar{z}=\phi(\tilde{z})$, with a $C^1$-function $\phi$ such that $\phi(0)=\phi^\prime(0)=0$. Then $z=\tilde{z}+ \phi(\tilde{z})\in \cPM$, hence $\cPM$ is locally parameterized over $\tilde{Z}$.

To prove surjectivity of this map, for given $(u,\vartheta,h)\in \cPM$,
solve problem \eqref{param1}, where the functions
$(F_d(u,h),G_\tau(u,\vartheta,h),G_\theta(\vartheta,h),  G_h(u,\vartheta,h))$
are now given.
By Proposition~\ref{param} the resulting linear problem has a unique solution $z=(u,\vartheta,h)$.
Let $\tilde{z}=z-\bar{z}$. Then we see that $\bar{z}=\phi(\tilde{z})$, hence the map
$[\tilde{z}\mapsto \tilde{z}+\phi(\tilde{z})]$ is also surjective near $0$.
\end{proof}
\medskip

\noindent
{\bf 3.} \,
Next we derive a similar decomposition for the solutions of
problem \eqref{enonlin-u-theta}.
Let $z_0=(\tilde{z}_0,\phi(\tilde{z}_0))\in\cSM$ be given, and let $z\in\EE(J)$ be the solution
of \eqref{enonlin-u-theta} with initial value $z_0$.
Then we would like to devise a decomposition of $z$ such that
$z(t)=\tilde{z}(t)+\bar{z}(t)$ with
$\tilde{z}(t)\in\tilde{Z}$ for $t\ge 0$.
As before, we use the notation $\tilde{z}=(\tilde{u},\tilde{\vartheta},\tilde{h})$,
and $\bar{z}=(\bar{u},\bar{\vartheta},\bar{h})$. In order to accomplish this we consider the
coupled systems of equations
\begin{equation}
\label{nonlin-param-bar}
\left\{
\begin{aligned}
\omega \bar{u}+\partial_t\bar{u} -\mu_*\Delta\bar{u} +\nabla\bar{\pi}&=F_u(u,\pi,\vartheta,h) &&\text{in} &&\Omega\setminus\Sigma\\
{\rm div}\, \bar{u}&= F_d(u,h) &&\text{in} && \Omega\setminus\Sigma\\
\bar{u}=\partial_{\nu} \bar{\vartheta} &=0 &&\text{on} &&\partial \Omega\\
[\![\bar{u}]\!]=[\![\bar{\vartheta}]\!]&=0 &&\text{on} &&\Sigma\\
-P_\Sigma[\![\mu_*(\nabla\bar{u} +[\nabla\bar{u}]^{\sf T})]\!]\nu_\Sigma &=G_\tau(u,\vartheta,h) &&\text{on} &&\Sigma\\
-\big([\![\mu_*(\nabla\bar{u} +[\nabla\bar{u}]^{\sf T}]\!]\nu_\Sigma|\nu_\Sigma\big) + [\![\bar{\pi}]\!]+\sigma \cA_\Sigma\bar{h}&=G_\nu(u,\vartheta,h)  &&\text{on} &&\Sigma\\
\kappa_*\omega\bar{\vartheta}+\kappa_*\partial_t\bar{\vartheta}-d_*\Delta \bar{\vartheta}
&=F_\theta(u,\vartheta,h)&&\text{in} &&\Omega\setminus\Sigma\\
l_* \bar{\vartheta} - \sigma\cA_\Sigma \bar{h} &=G_\theta(\vartheta,h) &&\text{on} &&\Sigma,\\
(l_*/\theta_*)\omega \bar{h}+ (l_*/\theta_*)(\partial_t \bar{h}-\bar{u}\cdot\nu_\Sigma)
-[\![d_*\partial_\nu \bar{\vartheta}]\!]&=G_h(u,\vartheta,h) &&\text{on} &&\Sigma \\
\bar{z}(0)&=\phi(\tilde{z}_0) && &&
\end{aligned}
\right.
\end{equation}
and
\begin{equation}
\label{nonlin-param-tilde}
\left\{
\begin{aligned}
 \partial_t \tilde{u}-\mu_*\Delta\tilde{u} +\nabla\tilde{\pi}&=\omega \bar{u} &&\text{in} &&\Omega\setminus\Sigma\\
{\rm div}\, \tilde{u}&= 0 &&\text{in} &&\Omega\setminus\Sigma\\
\tilde{u}=\partial_{\nu} \tilde{\vartheta} &=0 &&\text{on} && \partial \Omega\\
[\![\tilde{u}]\!]=[\![\tilde{\vartheta}]\!]&=0 &&\text{on} && \Sigma\\
-P_\Sigma[\![\mu_*(\nabla\tilde{u} +[\nabla\tilde{u}]^{\sf T}]\!]\nu_\Sigma &=0 &&\text{on} &&\Sigma\\
-\big([\![\mu_*(\nabla\tilde{u} +[\nabla\tilde{u}]^{\sf T}]\!]\nu_\Sigma|\nu_\Sigma\big)
+[\![\tilde{\pi}]\!]+\sigma \cA_\Sigma \tilde{h}
&=0 &&\text{on} && \Sigma\\
\kappa_*\partial_t\tilde{\vartheta}-d_*\Delta\tilde{\vartheta}&=\kappa_*\omega\bar{\vartheta} &&\text{in} &&\Omega\setminus\Sigma\\
l_* \tilde{\vartheta} - \sigma \cA_\Sigma \tilde{h}&= 0 &&\text{on} && \Sigma,\\
 (l_*/\theta_*)(\partial_t \tilde{h}-\tilde{u}\cdot\nu_\Sigma)
 -[\![d_*\partial_\nu \tilde{\vartheta}]\!]&=(l_*/\theta_*)\omega\bar{h} &&\text{on} &&\Sigma\\
\tilde{z}(0)&=\tilde{z}_0.
&& &&
\end{aligned}
\right.
\end{equation}
Equations \eqref{nonlin-param-bar}--\eqref{nonlin-param-tilde}
can be rewritten in the more condensed form
\begin{equation}
\label{condensed}
\begin{aligned}
\LL_{\omega} \bar{z}&= N(\tilde{z}+\bar{z}), \quad &&\bar{z}(0)=\phi(\tilde{z}_0)\\
\dot{\tilde{z}}+L\tilde{z}&=\omega \bar{z},
\quad &&\tilde{z}(0)=\tilde{z}_0,
\end{aligned}
\end{equation}
where we use the abbreviation $\mathbb{L}_{\omega}$ to denote the linear operator on
the left hand side of \eqref{nonlin-param-bar}, and $N$ to denote the nonlinearities
on the right hand side of \eqref{nonlin-param-bar}, respectively.

\begin{remark}
As ${\rm div}\,\bar{u}$ is in general nonzero, $\bar{z}$ does not belong to the base space $X_0$. However, this defect can be easily overcome, replacing
$\bar{u}$ in \eqref{nonlin-param-tilde} by its Helmholtz-projection in $\Omega$. This only changes the pressure $\tilde{\pi}$ by a jump-free part, but the velocity $\tilde{u}$, and hence also $\tilde{z}$, are unchanged.
\end{remark}
\medskip

\noindent
{\bf 4.}\, For the purpose of proving the stability result it turns out to be more convenient
to modify the decomposition of $z$ derived in the previous step in the following way.
Suppose $z_\infty=\tilde{z}_\infty+\phi(\tilde{z}_\infty)\in\cSM$ is an equilibrium of
\eqref{enonlin-u-theta} which is close to the fixed equilibrium $z_*=(0,\theta_*,\Gamma_*)$.
Then we decompose the solution $z$
of \eqref{enonlin-u-theta} as $z(t)=z_\infty+\tilde{z}(t)+\bar{z}(t)$,
 where as above $\tilde{z}(t)\in\tilde{Z}$.
Clearly, $\LL z_\infty=N(z_\infty)$,
and we are lead to consider the following coupled system for the pair $(\tilde{z},\bar{z})$
\begin{equation}
\label{condensed-inf}
\begin{aligned}
\LL_{\omega} \bar{z}&= N(z_\infty\! +\tilde{z}+\bar{z})-N(z_\infty),
\quad &&\bar{z}(0)=\phi(\tilde z_0)-\phi(\tilde{z}_\infty), \\
\dot{\tilde{z}}+L\tilde{z}
&=\omega \bar{z},
\quad &&\tilde{z}(0)=\tilde{z}_0-\tilde{z}_\infty.
\end{aligned}
\end{equation}
\medskip\\
\noindent
The abstract problem \eqref{condensed-inf}
can be treated in the same way as in the proof of
Theorem 5.2 in Pr\"uss, Simonett and Zacher \cite{PSZ10}.
This implies the following result.
\begin{theorem}
\label{stability} Let $p>n+2$, $\sigma>0$ and $l_*\neq0$,
and suppose $\psi_i\in C^3(0,\infty)$, $\mu_i,d_i\in C^2(0,\infty)$ are such that
$$\kappa_i(s)=-s\psi_i^{\prime\prime}(s)>0,\quad \mu_i(s)>0,\quad  d_i(s)>0,\quad s\in(0,\infty),\; j=1,2.$$
Let the function $\varphi$ be as in \eqref{energy-equilibrium}. Then in the topology of the state manifold $\cSM$ we have:
\begin{itemize}
\item[(a)]  $(0,\theta_*,\Gamma_*)\in\cE$ is stable if $\Gamma_*$ is connected and $\varphi^\prime(\theta_*)<0$.\\
Any solution starting in a neighborhood of such a stable equilibrium converges to another stable equilibrium exponentially fast.
\vspace{1mm}
\item[(b)] $(0,\theta_*,\Gamma_*)\in\cE$
is unstable if $\Gamma_*$ is disconnected or $\varphi^\prime(\theta_*)>0$.\\
Any solution starting and staying in a neighborhood of such an unstable equilibrium converges to another unstable equilibrium exponentially fast.
\end{itemize}
\end{theorem}
\section{Global Existence and Convergence}

We have seen in  \cite{PSSS11} that the negative total entropy, see \eqref{entropy},
is a strict Lyapunov functional.
Therefore the {\em limit sets} of solutions in the state manifold $\cSM$ are contained in the manifold $\cE\subset \cSM$ of equilibria.

There are several obstructions against global existence:
\begin{itemize}
\item
{\em Regularity}: the norms of either $u(t)$, $\theta(t)$, $\Gamma(t)$,
or  $[\![d(\theta(t))\partial_\nu \theta(t)]\!]$ may  become unbounded;
\item
{\em Well-posedness }: the  condition $l(\theta)\neq0$
may be violated; or the temperature may become $0$;
\item
{\em Geometry}: the topology of the interface may change;\\
    or the interface may touch the boundary of $\Omega$;\\
    or a part of the interface may contract to a point.
\end{itemize}
Recall that the compatibility conditions
\begin{align*}
&{\rm div}\, u(t)=0 \mbox{ in } \Omega\setminus\Gamma(t),
\quad u(t)=\partial_\nu\theta(t)=0 \mbox{ on } \partial\Omega,\\
&[\![u(t)]\!]=[\![\theta]\!]=P_\Gamma[\![\mu(\theta (t)) D(t)]\!]=0,
\quad[\![\psi(\theta(t))]\!]+\sigma H_\Gamma(t)=0 \mbox{ on }\Gamma(t),
\end{align*}
are preserved by the semiflow.

Let $(u,\theta,\Gamma)$ be a solution in the state manifold $\cSM$ with maximal interval of existence $[0,t_*)$. By the
{\em uniform ball condition} we mean the existence of a radius $r_0>0$ such that for each $t\in[0,t_*)$,
at each point $x\in\Gamma(t)$ there exists centers $x_i\in \Omega_i(t)$ such that
$B_{r_0}(x_i)\subset \Omega_i$ and $\Gamma(t)\cap \bar{B}_{r_0}(x_i)=\{x\}$, $i=1,2$. Note that this condition
bounds the curvature of $\Gamma(t)$, prevents parts of it to shrink to points, to touch the outer
boundary $\partial \Omega$, and to undergo topological changes.

With this property, combining the local semiflow for (\ref{i2pp}) with
the Lyapunov functional and compactness we obtain the following result.
\begin{theorem}
\label{Qual} Let $p>n+2$, $\sigma>0$, and suppose $\psi_i\in C^3(0,\infty)$, $\mu_i,d_i\in C^2(0,\infty)$ such that
$$\kappa_i(s)=-s\psi_i^{\prime\prime}(s)>0,\quad \mu_i(s)>0,\quad  d_i(s)>0,\quad s\in(0,\infty),\; i=1,2.$$
Suppose that $(u,\theta,\Gamma)$ is a solution of
(\ref{i2pp}) in the state manifold $\cSM$ on its maximal time interval $[0,t_*)$.
Assume there is constant $M>0$  such that the  following conditions hold on $[0,t_*)$:
\begin{itemize}
\item[(i)]  $|u(t)|_{[W^{2-2/p}_p]^n},|\theta(t)|_{W^{2-2/p}_p},|\Gamma(t)|_{W^{4-3/p}_p},
|[\![d(\theta(t))\partial_\nu \theta(t)]\!]|_{W^{2-6/p}_p}\leq M$;
\vspace{2mm}
\item[(ii)]  $|l(\theta(t))|, \theta(t)\geq 1/M$;
\vspace{2mm}
\item[(iii)] $\Gamma(t)$ satisfies the uniform ball condition.
\end{itemize}
Then $t_*=\infty$, i.e.\ the solution exists globally, and its limit set $\omega(u,\theta,\Gamma)\subset\cE$ is non-empty. If further $(0,\theta_\infty,\Gamma_\infty)\in\omega_+(u,\theta,\Gamma)$ with $\Gamma_\infty$ connected and $\varphi^\prime(\theta_\infty)<0$, then the solution converges in $\cSM$ to this equilibrium.

Conversely, if $(u(t),\theta(t),\Gamma(t))$ is a global solution in $\cSM$ which converges to an equilibrium $(0,\theta_*,\Gamma_*)\in\cE$  in $\cSM$ as $t\to\infty$, and $l(\theta_*)\neq0$,
then (i)--(iii) hold.
\end{theorem}
\begin{proof}
Under the assumptions (i)--(iii)
it is shown in the proof of \cite[Theorem 8.2]{PSSS11} that
$t_*=\infty$ and that the orbit $(u,\theta,\Gamma)(\R_+)\subset\cPM$ is relatively compact.
The negative total entropy is a strict Lyapunov functional, hence the limit set $\omega(u,\theta,\Gamma)\subset \cSM$ of a solution is contained in the set $\cE$ of equilibria.
By compactness, $\omega_+(u,\theta,\Gamma)\subset \cPM$ is non-empty,
hence the solution comes close to $\cE$, and stays there.
Then we may apply the convergence result Theorem \ref{stability}.
The converse follows by a compactness argument.
\end{proof}
\bigskip
\noindent
{\bf Remarks} \\
(i) \, We believe that in Theorem \ref{Qual} the assumption that $\Gamma_\infty$ is connected can be dropped and $\varphi^\prime(\theta_\infty)<0$  can be replaced by $\varphi^\prime(\theta_\infty)\neq0$. However, a proof of this requires much more technical efforts, we refrain from these, here.

\medskip

\noindent
(ii) \,
We cannot show that the temperature stays positive if it is initially since we did not make any assumptions on the behavior of the functions $\mu_j,d_j,\psi_j$ near $0$.


{\small
\begin{thebibliography}{99}

\bibitem{Gur07} D.M. Anderson, P. Cermelli, E. Fried, M.E. Gurtin, G.B. McFadden,
General dynamical sharp-interface conditions for phase transformations in viscous heat-conducting fluids.
{\em J. Fluid Mech.} {\bf 581} (2007), 323--370.


\bibitem{DBFr86}
E.~DiBenedetto, A.~Friedman, Conduction-convection problems with change of
  phase, \emph{J. Differential Equations} \textbf{62} (1986), no.~2, 129--185.

\bibitem{DBOL93}
E.~DiBenedetto, M.~O'Leary, Three-dimensional conduction-convection problems
  with change of phase, \emph{Arch. Rational Mech. Anal.} \textbf{123} (1993),
  no.~2, 99--116.

\bibitem{DHP03} R. Denk, M. Hieber, and J. Pr\"uss,
{\it ${\mathcal R}$-boundedness, Fourier multipliers, and problems of
elliptic and parabolic type}, AMS Memoirs 788, Providence, R.I. (2003).

\bibitem{DHP07}R. Denk, M. Hieber, J.Pr\"uss,
Optimal $L\sp p$-$L\sp q$-estimates for parabolic boundary value problems with inhomogeneous data.
{\em Math. Z.} {\bf 257}  (2007),  no. 1, 193--224.

\bibitem{ES98} J. Escher, G. Simonett,
A center manifold analysis for the Mullins-Sekerka model.
{\em J. Differential Equations} {\bf 143} (1998), 267--292.

\bibitem{HoSt98a}
K.-H.~Hoffmann, V.N.~Starovoitov, The {S}tefan problem with surface tension and convection in {S}tokes
  fluid, \emph{Adv. Math. Sci. Appl.} \textbf{8} (1998), no.~1, 173--183.

\bibitem{HoSt98b}
K.-H.~Hoffmann, V.N.~Starovoitov, Phase transitions of liquid-liquid type
  with convection, \emph{Adv. Math. Sci. Appl.} \textbf{8} (1998), no.~1,
  185--198.

\bibitem{Ish75} M. Ishii, {\em Thermo-Fluid Dynamic Theory of Two-Phase Flow}
Collection de la Direction des \'Etudes et Recherches D'\'Electricit\'e d France, Paris 1975.

\bibitem{IsTa06} M. Ishii and H. Takashi,
{\em Thermo-Fluid Dynamics of Two-Phase Flow,} Springer, New York, 2006.


\bibitem{KPW10} M.~K\"ohne, J.~Pr\"uss, M.~Wilke, Qualitative behaviour of solutions for the
two-phase Navier-Stokes equations with  surface tension.
{\em Math. Ann.} (To appear 2012).

\bibitem{Kus02}
Y.~Kusaka, On a limit problem of the {S}tefan problem with surface tension in a
  viscous incompressible fluid flow, \emph{Adv. Math. Sci. Appl.} \textbf{12}
  (2002), no.~2, 665--683.

\bibitem{KuTa99}
Y.~Kusaka, A.~Tani, On the classical solvability of the {S}tefan problem in
  a viscous incompressible fluid flow, \emph{SIAM J. Math. Anal.} \textbf{30}
  (1999), no.~3, 584--602 (electronic).

\bibitem{KuTa02}
Y.~Kusaka, A.~Tani, Classical solvability of the two-phase {S}tefan problem in a viscous
  incompressible fluid flow, \emph{Math. Models Methods Appl. Sci.} \textbf{12}
  (2002), no.~3, 365--391.

\bibitem{Pru03} J.~Pr\"uss, Maximal regularity for evolution equations in $L_p$-spacess.
{\em Conf. Sem. Mat. Univ. Bari} {\bf 285}, 1--39 (2003)

\bibitem{PSSS11} J.~Pr\"uss, Y.~Shibata, S.~Shimizu, G.~Simonett,
\newblock On well-posedness of incompressible two-phase flows with phase transitions: The case of equal densities.
\newblock {\em Evolution Equations and Control Theory} {\bf 1} (2012), 171--194.


\bibitem{PrSi08} J.~Pr\"uss, G.~Simonett,
Stability of equilibria for the Stefan problem with surface tension.
{\em SIAM J.~Math.~Anal.}{\bf 40} (2008), 675--698.


\bibitem{PSZ09} J. Pr\"uss, G. Simonett, R. Zacher,
Convergence of solutions to equilibria for nonlinear
parabolic problems. {\em J.~Diff.~Equations} {\bf 246} (2009), 3902--3931.

\bibitem{PSZ10} J. Pr\"uss, G. Simonett, R. Zacher,
\newblock Qualitative behaviour of the solutions for Stefan problems with surface tension.
\newblock {\em arXiv:1101.3763.} To appear in {\em Arch. Ration. Mech. Anal.}


\bibitem{Tanaka93} N. Tanaka,
Two-phase free boundary problem for viscous incompressible
thermo-capillary convection.
{\em Japan J. Mech.} {\bf 21} (1995), 1--41.

\end{thebibliography}}
\end{document}